\documentclass{amsart}

\usepackage{amsmath,amsfonts,amssymb,amscd,verbatim}

\def\Q{{\mathbb Q}}

\def\fchar{\mathrm{char}}

\newtheorem{thm}{Theorem}[section]
\newtheorem{lem}[thm]{Lemma}
\newtheorem{cor}[thm]{Corollary}
\newtheorem{prop}[thm]{Proposition}

\theoremstyle{definition}

\newtheorem{ex}[thm]{Example}

           \newtheorem{rem}[thm]{Remark}

\title{The divisibility by $2$ of rational points on elliptic curves}
\author {Boris M. Bekker}
\address{St. Petersburg State University, Department of Mathematics and Mechanics,   Universitetsky prospekt, 28, Peterhof, St. Petersburg, 198504, Russia8}
\email{ bekker.boris@gmail.com}
\thanks{The first named author (B.B.) is partially supported by RFFI grant N 14-01-00393}
\author {Yuri G. Zarhin}
\address{Pennsylvania State University, Department of Mathematics, University Park, PA 16802, USA}
\email{zarhin@math.psu.edu}
\thanks{The second named author (Y.Z.) is  partially supported by a grant from the Simons Foundation (\#246625 to Yuri Zarkhin).
Part of this work was done  in May-June 2016 when he was a visitor at the Max-Planck-Institut f\"ur Mathematik (Bonn, Germany),
whose hospitality and support are gratefully acknowledged.}

\begin{document}\date{}

\begin{abstract}
We give a simple proof of the well-known divisibility by 2  condition for rational points on elliptic curves with
rational 2-torsion. As an application of the explicit division by $2^n$ formulas obtained in Sec.2, we construct versal families of elliptic curves containing points of orders 4, 5,  6, and 8 from which we obtain an explicit description of elliptic curves over certain finite fields $\mathbb F_q$ with a prescribed (small) group $E(\mathbb F_q)$. In the last two sections we study 3- and 5-torsion.
\end{abstract}
\maketitle

\section{Introduction}
Let $E$ be an elliptic curve over a number field $K$. A famous {\sl Mordell-Weil theorem} asserts that the (abelian)  group $E(K)$ of $K$-points on $E$  is finitely generated \cite{Cassels,SilvermanTate,Wash}. The first step in its proof (and actual finding a finite set that generates $E(K)$)
% \cite[Sect. 11]{Manin})
 is a {\sl weak Mordell-Weil theorem} that asserts that the quotient $E(K)/2 E(K)$ is a finite (abelian) group. This step is called 2-descent  and its basic ingredient is a criterion for when a $K$-point  on $E$ is twice another $K$-point (under an additional assumption that all points of order 2 on $E$ are defined over $K$).
In this paper we give a new treatment of this criterion that seems to be less computational than previous ones (\cite[Ch. 5, pp. 102--104]{Lang},  \cite{Husemoller}, \cite[Th. 4.2 on pp. 85-87]{Knapp}, \cite[Lemma 7.6 on p. 67]{Buhler}
\cite[pp. 331--332]{Bombieri}).  This approach allows us to describe explicitly 2-power torsion on  elliptic curves.  In addition we obtain explicitly families of elliptic curves with various torsion subgroups
 over arbitrary fields of characteristic different from 2 (the problem of constructing elliptic curves with given torsion goes back to B.Levi \cite {SS}).

The paper is organized as follows. We work with elliptic curves $E$ over an arbitrary field $K$ with $\fchar(K)\ne 2$.
In Section \ref{l1} we discuss the criterion of divisibility by 2 and explicit formulas for the ``half-points'' in $E(K)$.
Next   we discuss a criterion of divisibility by any power of $2$ in $E(K)$ (Section \ref{power2}). In Section \ref{torsion} we collect useful results about elliptic curves and their torsion.
In Sections \ref{l4},\ref{l8}, and \ref{l6}  we will use explicit formulas of Section \ref{l1} in order to construct {\sl versal}  families of elliptic curves $E$ such that $E(K)$ contains a subgroup isomorphic to $\mathbb{Z}/2m\mathbb{Z}\oplus \mathbb{Z}/2\mathbb{Z}$ with $m=2,4,3$, respectively. (In addition, in Section  \ref{l4} we construct a {\sl versal}  family of elliptic curves $E$ such that $E(K)$ contains a subgroup isomorphic to $\mathbb{Z}/4\mathbb{Z}\oplus \mathbb{Z}/4\mathbb{Z}$.)
 Such families are parameterized by $K$-points of rational curves that are closely related to certain modular curves of genus zero (see \cite{SS,Kubert,Silver,Silver2}); however, our approach remains quite elementary. In addition, in Sections \ref{l8} and \ref{l5} we construct {\sl versal}  families of elliptic curves $E$ such that $E(K)$ contains a subgroup isomorphic to $\mathbb{Z}/8\mathbb{Z}\oplus \mathbb{Z}/4\mathbb{Z}$ and $\mathbb{Z}/10\mathbb{Z}\oplus \mathbb{Z}/2\mathbb{Z}$, respectively.  These two families are parameterized by $K$-points  of   curves  that are closely related to certain modular curves of genus 1.

  As an unexpected application, we describe explicitly (and without computations) elliptic curves $E$ over  small finite fields $\mathbb{F}_q$  such that $E(\mathbb{F}_q)$ is isomorphic to  a certain finite group (of small order).  Using deep highly nontrivial results of B.  Mazur \cite{Mazur} and of S. Kamienny and M. Kenku--F. Momose \cite{Kam,Ken}, we describe explicitly elliptic curves $E$ over the field  $\Q$ of rational numbers and over quadratic fields $K$ such that the torsion subgroup $E(\Q)_t$ of $E(\Q)$ (resp.  $E(K)_t$ of $E(K)$) is isomorphic to  a certain finite group.

{\bf Acknowledgements}.  We are grateful to Robin Chapman for helpful comments.
Our special thanks go to Tatiana Bandman for  help with \textbf{magma}.

\section{Division by 2}
\label{l1}

Let $K$ be a field of characteristic different from $2$.
Let
\begin{equation}
\label{E2}
E: y^2=(x-\alpha_1)(x-\alpha_2)(x-\alpha_3)
\end{equation}
be an elliptic curve over $K$, where $\alpha_1,\alpha_2,\alpha_3$ are {\sl distinct} elements of $K$.
This means that $E(K)$ contains all three points of order 2, namely, the points
\begin{equation}
\label{W2}
W_1=(\alpha_1,0), W_2=(\alpha_2,0), W_3=(\alpha_3,0).
\end{equation}
The following statement is pretty well known (\cite[pp. 269--270]{Cassels},
\cite[Ch. 5, pp. 102--104]{Lang},  \cite{Husemoller}, \cite[Th. 4.2 on pp. 85-87]{Knapp}, \cite[Lemma 7.6 on p. 67]{Buhler}
\cite[pp. 331--332]{Bombieri}, \cite[pp. 212--214]{Wash}; see also \cite{Yelton}).
\begin{thm}
\label{th0}
Let $P=(x_0,y_0)$ be a $K$-point on $E$. Then $P$ is divisible by $2$ in $E(K)$ if and only if all three elements $x_0-\alpha_i$ are squares in $K$.
\end{thm}
While the proof of the claim that the divisibility implies the squareness is straightforward,
it seems that the known elementary  proofs of the converse statement are
more involved/computational. (Notice that there is another approach, which is based on Galois cohomology \cite[Sect. X.1, pp. 313--315]{Silverman}
and it works for hyperelliptic jacobians as well \cite{Schaefer}.)

We start with an elementary proof of the divisibility that seems to be less computational. (In additional, it will give us immediately explicit formulas for the coordinates of all four $\frac{1}{2}P$.)

\begin{proof}
So, let us assume that all three elements $x_0-\alpha_i$ are squares in $K$, and let $Q=(x_1,y_1)$ be a point on $E$ with $2 Q=P$.  Since $P\ne \infty$, we have $y_1\ne 0$, and therefore
the equation of the {\sl tangent line} $L$ to $E$ at $Q$ may be written in the form
$$L: y=l x+m.$$
(Here $x_1,y_1, l, m$ are elements of an  overfield of $K$.) In particular,
$y_1=l x_1+m$.
By the definition of $Q$ and $L$, the point $-P=(x_0,-y_0)$ is the ``third'' common point of $L$ and $E$;
in particular,
$-y_0=l x_0+m$,
i.e., $y_0=-(l x_0+m)$.
Standard arguments (the restriction of the equation for $E$ to $L$, see \cite[pp. 25--27]{SilvermanTate}, \cite[pp. 12--14]{Wash}, \cite[p. 331]{Bombieri}) tell us that
the monic cubic polynomial
$$(x-\alpha_1)(x-\alpha_2)(x-\alpha_3)-(l x+m)^2$$
coincides with $(x-x_1)^2 (x-x_0)$. This implies that
$$-(l \alpha_i+m)^2=(\alpha_i-x_1)^2 (\alpha_i-x_0) \ \text{for all}\ i=1,2,3.$$
Since $2Q=P\ne \infty$, none of   $x_1 -\alpha_i$ vanishes. Recall that all  $x_0-\alpha_i$ are squares in $K$ and they are obviously distinct. Consequently,  the
 corresponding square roots \cite[p. 331]{Bombieri}
\begin{equation*}
\label{alphaR}
r_i:=\frac{l\alpha_i+m}{x_1 - \alpha_i}=\sqrt{x_0-\alpha_i}
\end{equation*}
are {\sl distinct} elements of $K$. In other words, the transformation
$$z \mapsto \frac{l z+m}{-z+x_1}$$ of the projective line
sends the three distinct $K$-points $\alpha_1,\alpha_2,\alpha_3$ to the three distinct
$K$-points $r_1,r_2,r_3$, respectively.
%(They are distinct, because all $\alpha_i$ are distinct and therefore all $x_1 -\alpha_i$ are distinct .)
This implies that our transformation is {\sl not} constant, i.e., is an honest linear fractional transformation
\footnote
{Another way to see this is to assume the contrary.  Then the {\sl determinant} $l x_1+m=0$, i.e., $y_0=0$, and therefore $P=2Q$ is the infinite point, which is not true.}
and is defined over $K$.
Since one of the ``matrix entries",
$-1$, is already a nonzero element of $K$,  all other matrix entries $l, m, x_1$ also lie in $K$. Since $y_1=l x_1 +m$, it also lies in $K$.
So, $Q=(x_1,y_1)$ is a $K$-point of $E$.
\end{proof}

Let us get explicit formulas for $x_1,y_1, l, m$ in terms of $r_1,r_2,r_3$. We have
$$\alpha_i=x_0-r_i^2, \ l\alpha_i+m=r_i (x_1 - \alpha_i),$$
and therefore
$$l (x_0-r_i^2)+m=r_i [x_1-(x_2-r_i^2)]=r_i^3+(x_1-x_2)r_i,$$
which is equivalent to
$r_i^3+ l r_i^2+(x_1-x_0)r_i-(l x_0+m)=0$,
and this equality holds for all $i=1,2,3$. This means that the monic cubic polynomial
\begin{equation*}
\label{polyH}
h(t)=t^3+l t^2+(x_1-x_0)t-(l x_0+m)
\end{equation*}
coincides with $(t-r_1)(t-r_2)(t-r_3)$. Recall that $-(l x_0+m)=y_0$ and get
\begin{equation}
\label{product}
r_1 r_2 r_3=-y_0.
\end{equation}
We also get
\begin{equation*}
\label{slope}
l=-(r_1+r_2+r_3), \ x_1-x_0=r_1 r_2+r_2 r_3+r_3 r_1.
\end{equation*}
This implies that
\begin{equation}
\label{x1}
x_1=x_0+(r_1 r_2+r_2 r_3+r_3 r_1).
\end{equation}
Since $y_1=l x_1+m$ and $-y_0=l x_0+m$, we obtain that
$$m=-y_0-l x_0=-y_0+(r_1+r_2+r_3)x_0,$$
and therefore
$$y_1= -(r_1+r_2+r_3)[x_0+(r_1 r_2+r_2 r_3+r_3 r_1)]+[-y_0+(r_1+r_2+r_3)x_0],$$ i.e.,
\begin{equation}
\label{y1}
y_1=-y_0-(r_1+r_2+r_3)(r_1 r_2+r_2 r_3+r_3 r_1).
\end{equation}
Notice that there are precisely four points $Q \in E(K)$ with $2Q=P$,
\begin{equation}
\label{halfP}
Q=\left(x_0+(r_1 r_2+r_2 r_3+r_3 r_1),-y_0-(r_1+r_2+r_3)(r_1 r_2+r_2 r_3+r_3 r_1)\right),
\end{equation}
each of which corresponds to one of the {\sl four} choices of the three square roots $r_i=\sqrt{x_0-\alpha_i}\in K$ ($i=1,2,3$) with
$r_1 r_2 r_3=-y_0$.
Using the latter equality, we may rewrite \eqref{y1} as
\footnote{This  was brought to our attention  by Robin Chapman.}
\begin{equation}
\label{chap}
y_1=-(r_1+r_2)(r_2+r_3)(r_3+r_1).
\end{equation}
In addition,
\begin{equation}
\label{x1prod}
x_1=\alpha_i+(r_i+r_j)(r_i+r_k),
\end{equation}
where $i,j,k$ is any permutation of $1,2,3$. Indeed,
$$x_1-\alpha_i=(x_0-\alpha_i)+r_1 r_2+r_2 r_3+r_3 r_1=$$
$$r_i^2+r_1 r_2+r_2 r_3+r_3 r_1=(r_i+r_j)(r_i+r_k).$$
The remaining four choices of the ``signs" of $r_1,r_2,r_3$ bring us to  the same values of abscissas and the opposite values of ordinates and give the results of division by 2 of the point $-P$.

Conversely, if we know $Q=(x_1,y_1)$, then we may recover the corresponding $(r_1,r_2,r_3)$. Namely, the equalities (\ref{x1prod}) and (\ref{chap}) imply that
\begin{equation*}
\label{QtoR}
\begin{aligned}
r_j+r_k=-\frac{y_1}{x_1-\alpha_i},\\
r_i=\frac{-(r_j+r_k)+(r_i+r_j)+(r_i+r_k)}{2}\\=-\frac{y_1}{2}\cdot
\left(-\frac{1}{x_1-\alpha_i}+\frac{1}{x_1-\alpha_j}+\frac{1}{x_1-\alpha_k}\right)
\end{aligned}
\end{equation*}
for any permutation $i,j,k$ of $1,2,3$.
%%%%%%%%%%%%%%%%%%%

\begin{ex}
\label{halfW3}
 Let us choose as $P=(x_0,y_0)$ the point $W_3=(\alpha_3,0)$  of order $2$ on $E$. Then $r_3=0$, and we have two arbitrary independent choices of  (nonzero) $r_1=\sqrt{\alpha_3-\alpha_1}$ and $r_2=\sqrt{\alpha_3-\alpha_2}$.
 Thus
 $$Q=(\alpha_3+r_1 r_2, -(r_1+r_2)r_1 r_2)=(\alpha_3+r_1 r_2, -r_1(\alpha_3-\alpha_2)-r_2(\alpha_3-\alpha_1))$$
 is a point on $E$ with $2Q=P$; in particular, $Q$ is a point of order $4$. The same is true for the (three remaining) points
 $-Q=(\alpha_3+r_1 r_2, r_1(\alpha_3-\alpha_2)+r_2(\alpha_3-\alpha_1))$,
 \newline
 $(\alpha_3-r_1 r_2, - r_1(\alpha_3-\alpha_2)+r_2(\alpha_3-\alpha_1))$, and
 $(\alpha_3-r_1 r_2,  r_1(\alpha_3-\alpha_2)-r_2(\alpha_3-\alpha_1))$.
\end{ex}

%%%%%%%%%%%%%%%%%%%%%%%%
Recall that, in formula (\ref{halfP}) for the coordinates of the points $\frac{1}{2}{P}$, one may arbitrarily choose the signs of $r_1,r_2,r_3$ under  condition \eqref{product}.
Let $Q$ be one of $\frac{1}{2}{P}$'s that corresponds to a certain choice of $r_1,r_2, r_3$. The remaining three {\sl halves} of $P$ correspond to
$(r_1,-r_2, -r_3)$, $(-r_1,r_2, -r_3)$, $(-r_1,-r_2, r_3)$. Let us denote these halves by $\mathcal{Q}_1, \mathcal{Q}_2, \mathcal{Q}_3$, respectively. For each $i=1,2,3$, the difference $\mathcal{Q}_i-Q$ is a point of order 2 on $E$. Which one? The following assertion answers this question.

\begin{thm}
\label{sign}
Let $i,j,k$ be a permutation of $1,2,3$. Then

\begin{itemize}
\item[(i)]
If $P=W_i$, then $\mathcal{Q}_i= -Q$.
\item[(ii)]
If $P \ne W_i$, then all three points $\mathcal{Q}_i, -Q, W_i$ are distinct.
\item[(iii)]
The points $\mathcal{Q}_i, -Q, W_i$ lie on the line
$$y=(r_j+r_k)(x-\alpha_i).$$
\item[(iv)]
$\mathcal{Q}_i-Q=W_i.$
\end{itemize}

\end{thm}

\begin{proof}
 First, assume that  $P=W_i$. In this case, formulas \eqref{x1}  and  \eqref{y1}  tell us that
$$Q=(\alpha_i+r_j r_k, - r_j r_k(r_j+r_k)),$$
which implies that
$$\mathcal{Q}_i=(\alpha_i+r_j r_k, r_j r_k(r_j+r_k))=-Q$$
and $$\mathcal{Q}_i-Q=-2Q=-P=P=W_i.$$
This proves (i) and a special case of (iv) when $P=W_i$.
\begin{comment}
Since $\mathcal{Q}_i\ne \mathcal{Q}_j$, we conclude that
$$\mathcal{Q}_j \ne -Q, \  \mathcal{Q}_k \ne Q$$
(recall that $i,j,k$ is a permutation of $1,2,3$).
\end{comment}
Now assume that $P \ne W_i$ and prove  that the three points $\mathcal{Q}_i, -Q, W_i$  are {\sl distinct}. Since none of $Q_i$ and $-Q$ is of order $2$,   none of them is $W_i$. On the other hand, if $\mathcal{Q}_i= -Q$, then
$$2Q=P=2\mathcal{Q}_i=-2Q=-P,$$
and so $P$ has order $2$, say $P=W_j$.  Applying (a) to $j$ instead of $i$, we get $\mathcal{Q}_j= -Q$; but $\mathcal{Q}_i\ne\mathcal{Q}_j$ since $i\ne j$.  Therefore $\mathcal{Q}_i, -Q, W_i$ are three {\sl distinct} points.
This proves (ii).

Let us prove (iii).
  Since
$$x_1-\alpha_i=(r_i+r_j)(r_i+r_k), \ y_1=-(r_1+r_2)(r_2+r_3)(r_3+r_1),$$
%(where $i,k,l$ is a permutation of $1,2,3$
we have $y_1=(r_j+r_k)(x_1-\alpha_i)$. Further
$$x(-\mathcal{Q}_i)-\alpha_i=(r_i-r_j)(r_i-r_k),$$
$$y(-\mathcal{Q}_i)=(r_i-r_j)(-r_j-r_k)(-r_k+r_i)=(r_j+r_k)\left(x(-\mathcal{Q}_i)-\alpha_i\right).$$
Therefore $\mathcal{Q}_i, -Q, W_i$ lie on the line
$$y=(r_j+r_k)(x-\alpha_i).$$

We have already  proven  (iv) when $P=W_i$. So, let us assume that $P \ne W_i$. Now (iv) follows from (iii) combined with (i).
\end{proof}

%%%%%new stuff added on July 17, 2017%%%%%

%%%%%%%%%%%%%%%%%%%%%%%%%%%%%%%%%%%%%%%%%%%%%
\section{Division by  $2^n$}
\label{power2}
Using the formulas above that describe the division by 2 on $E$, one may easily deduce the following necessary and sufficient condition of divisibility by any power of 2. For an overfield $L$ of $K$, we consider a sequence  of points $Q_{\mu}$ in $E(L)$  such that $Q_0=P$ and $2 Q_{\mu+1}=Q_{\mu}$ for all $\mu=0,1,2, \dots $. Let $r_1^{(\mu)}, r_2^{(\mu)}, r_3^{(\mu)}$ ($\mu=0,1,2, \dots $) be arbitrary sequences of elements of $L$ that satisfy the relations
$$(r_i^{(\mu)})^2=x(Q_{\mu})-\alpha_i.$$
Then for each permutation $i,j,k$ of $1,2,3$ we obtain, in light of the formula (\ref{x1prod}),
$$  x(Q_{\mu+1})-\alpha_i=
\bigl(r_i^{(\mu)}+r_j^{(\mu)}\bigr)\bigl(r_i^{(\mu)}+r_k^{(\mu)}\bigr),
$$
which implies that
$$(r_i^{(\mu+1)})^2=(r_i^{(\mu)}+r_j^{(\mu)})(r_i^{(\mu)}+r_k^{(\mu)}).$$
By changing the signs of  $r_i^{(\mu)}, r_j^{(\mu)}, r_k^{(\mu)}$ in the product $(r_i^{(\mu)}+r_j^{(\mu)})(r_i^{(\mu)}+r_k^{(\mu)})$, we obtain all possible values of the abscissas of $Q_{(\mu+1)}$ with $2Q_{\mu+1}=Q_{\mu}$.

Suppose that $Q_{\mu}\in E(K)$.  Then $Q_{\mu}$ is divisible by $2$ in $E(K)$ if and only if one may choose $r_i^{(\mu)}, r_j^{(\mu)}, r_k^{(\mu)}$ in such a way that $(r_i^{(\mu)}+r_j^{(\mu)})(r_i^{(\mu)}+r_k^{(\mu)})$ are squares in $K$ for all $i=1,2,3$. We proved the following statement.

\begin{thm}
\label{divisionByPower}
Let $P=(x_0,y_0)\in E(K)$. Let $r_1^{(\mu)}, r_2^{(\mu)}, r_3^{(\mu)}$ $($$\mu=0,1,2, \dots$ $)$ be  sequences of elements of $L$ that satisfy the relations
$$(r_i^{0})^2=r_i^2=x_0-\alpha_i, \ (r_i^{(\mu+1)})^2=(r_i^{(\mu)}+r_j^{(\mu)})(r_i^{(\mu)}+r_k^{(\mu)})$$
for all permutations $i,j,k$ of $1,2,3$.
Then $P$ is divisible by $2^n$ in $E(K)$ if and only if all $x_0-\alpha_i$ are squares in $K$, and, for each $\mu=0,1, \dots n-1$, one may choose square roots $r_1^{(\mu)}, r_2^{(\mu)}, r_3^{(\mu)}$ in such a way that the products $(r_i^{(\mu)}+r_j^{(\mu)})(r_i^{(\mu)}+r_k^{(\mu)})$  are squares in $K$ {\rm(}and therefore all $r_i^{(\mu)}$ lie in $K$ for  $\mu=0,1, \dots n-1${\rm)}.
\end{thm}

The knowledge of sequences $r_1^{(\mu)}, r_2^{(\mu)}, r_3^{(\mu)}$ allows us step by step to find the points $\frac{1}{2}P,  \frac{1}{4}P, \frac{1}{8}P$ etc.

\begin{ex}
Let $P=(x_0,y_0)$,   let $R$ be a  point of $E$ such that $4R=P$, and let $Q=2R=(x_1,y_1)$.  By formulas (\ref{x1}) and (\ref{chap}),
$$x_1=x_0+(r_1 r_2+r_2 r_3+r_3 r_1), \ y_1=-(r_1+r_2)(r_2+r_3)(r_3+r_1),$$
where the square roots
$$r_i=\sqrt{x_0-\alpha_i}, \ i=1,2,3,$$
are chosen in such a way that $r_1 r_2 r_3=-y_0$. Further, let
$$r_i^{(1)}=\sqrt{(r_i+r_j)(r_i+r_k)}$$
be square roots that are chosen in such a way that
$$r_1^{(1)} r_2^{(1)} r_3^{(1)}=-y_1=(r_1+r_2)(r_2+r_3)(r_3+r_1).$$
In light of (\ref{x1}) and (\ref{chap}),
$$x(R)=x_1+r_1^{(1)} r_2^{(1)}+r_2 ^{(1)}r_3^{(1)}+r_3^{(1)} r_1^{(1)},$$
$$y(R)=-(r_1^{(1)}+r_2^{(1)})(r_2^{(1)}+r_3^{(1)})(r_3^{(1)}+r_1^{(1)}),$$
which implies that
\begin{equation}\label{1/4}
 \begin{aligned}x(R)=x_0+(r_1 r_2+r_2 r_3+r_3 r_1)+(r_1^{(1)} r_2^{(1)}+r_2 ^{(1)}r_3^{(1)}+r_3^{(1)} r_1^{(1)}),\\
y(R)=-(r_1^{(1)}+r_2^{(1)})(r_2^{(1)}+r_3^{(1)})(r_3^{(1)}+r_1^{(1)}).\end{aligned}
\end{equation}
\end{ex}

%%%%%%%%%%%%%%%%%%%%%%%%%%%%%%%%%%%%%
\section{Torsion of elliptic curves}
\label{torsion}
In the sequel, we will freely use the following well-known elementary observation.

{\sl Let $\kappa$ be a nonzero element of $K$. Then there is a canonical isomorphism of the elliptic curves
$$E: y^2=(x-\alpha_1)(x-\alpha_2)(x-\alpha_3)$$
 and
$$E(\kappa): {y^{\prime}}^2=\left(x^{\prime}-\frac{\alpha_1}{\kappa^2}\right)
\left(x^{\prime}-\frac{\alpha_2}{\kappa^2}\right)\left(x^{\prime}-\frac{\alpha_3}{\kappa^2}\right)$$
that is given by the change of variables
$$x^{\prime}=\frac{x}{\kappa^2}, \ y^{\prime}=\frac{y}{\kappa^3}$$
%%%%%%%%%%%added on August 5, 2016%%%%%%%%%%%%%%%%
and respects the group structure.
Under this isomorphism the point
$(\alpha_i,0)\in E(K)$ goes to $(\alpha_i / \kappa^2 ,0)\in E(\kappa)(K)$ for all $i=1,2,3$.
In addition, if $P=(0,y(P))$  lies in $E(K)$,
then it goes (under this isomorphism) to $(0,y(P)/\kappa^3)\in E(\kappa)(K)$.}

We will also use the following classical result of Hasse (Hasse bound)
\cite[Th. 4.2 on p. 97]{Wash}.
\begin{thm}
\label{hasse}
 If $q$ is a prime power, $\mathbb{F}_q$ a $q$-element finite field and $E$ is an elliptic curve over $\mathbb{F}_q$, then $E(\mathbb{F}_q)$ is a finite abelian group whose cardinality $|E(\mathbb{F}_q)|$ satisfies the inequalities
\begin{equation}
\label{HasseB}
q-2\sqrt{q}+1 \le |E(\mathbb{F}_q)|\leq q+2\sqrt{q}+1.
\end{equation}
\end{thm}

Another result that we are going to use is the following immediate corollary of a celebrated theorem of B. Mazur  (\cite{Mazur},  \cite[Th. 2.5.2 and p. 187]{Robledo}).

\begin{thm}
\label{mazurQ}
If $E$ is an elliptic curve over $\Q$ and the torsion subgroup $E(\Q)_t$ of $E(\Q)$ is not cyclic, then
$E(\Q)_t$ is isomorphic to $\mathbb{Z}/2m\mathbb{Z}\oplus \mathbb{Z}/2\mathbb{Z}$ with $m=1,2,3$ or $4$.
In particular, if $m=3$ or $4$ and $E(\Q)$ contains a subgroup isomorphic to $\mathbb{Z}/2m\mathbb{Z}\oplus \mathbb{Z}/2\mathbb{Z}$,
then $E(\Q)_t$ is isomorphic to $\mathbb{Z}/2m\mathbb{Z}\oplus \mathbb{Z}/2\mathbb{Z}$.
\end{thm}

The next assertion  follows readily from the list of possible torsion subgroups of elliptic curves over quadratic fields   obtained by S. Kamienny \cite{Kam} and M.A. Kenku - F. Momose \cite{Ken} (see also \cite[Th. 1]{KamN}).

\begin{thm}
\label{Kquad}
Let $E$ be an elliptic curve over a quadratic field $K$. Assume that all points of order 2 on $E$ are defined over $K$.
Let $E(K)_t$ be the torsion subgroup of $E(K)$. Then  $E(K)_t$ is isomorphic either to $\mathbb{Z}/4\mathbb{Z}\oplus \mathbb{Z}/4\mathbb{Z}$
or to  $\mathbb{Z}/2m\mathbb{Z}\oplus \mathbb{Z}/2\mathbb{Z}$ with $1\le m\le 6$.

In particular,
$E(K)_t$  enjoys the following properties.

\begin{enumerate}
\item
If $m=5$ or $6$ and  $E(K)$ contains a subgroup  isomorphic to  $\mathbb{Z}/2m\mathbb{Z}\oplus \mathbb{Z}/2\mathbb{Z}$, then
$E(K)_t$ is  isomorphic to  $\mathbb{Z}/2m\mathbb{Z}\oplus \mathbb{Z}/2\mathbb{Z}$.
\item
If  $E(K)$ contains a subgroup  isomorphic to  $\mathbb{Z}/4\mathbb{Z}\oplus \mathbb{Z}/4\mathbb{Z}$, then
$E(K)_t$ is  isomorphic to   $\mathbb{Z}/4\mathbb{Z}\oplus \mathbb{Z}/4\mathbb{Z}$.
\end{enumerate}
\end{thm}

\section{Rational points of order 4}
\label{l4}

We are going to describe explicitly elliptic curves  (\ref{E2}) that contain a $K$-point of order $4$. In order to do that,
we consider the elliptic curve
$$\mathcal{E}_{1,\lambda}: y^2=(x+\lambda^2)(x+1)x$$
 over $K$. Here $\lambda$ is an element of $K\setminus \{0, \pm 1\}$.
 In this case, we have
$$\alpha_1=-\lambda^2,\ \alpha_2=-1,\ \alpha_3=0.$$
Notice that
$$\mathcal{E}_{1,\lambda}=\mathcal{E}_{1,-\lambda}.$$
All three differences
$$\alpha_3-\alpha_1=\lambda^2,\ \alpha_3-\alpha_2=1^2,\  \alpha_3 - \alpha_3=0^2$$
are squares in $K$. Dividing the order $2$ point $W_3=(0,0)\in \mathcal{E}_{1,\lambda}(K)$ by   $2$,
we get $r_3=0$ and the four choices
$$r_1=\pm \lambda,\ r_2=\pm 1.$$
Now  Example \ref{halfW3}  gives us four points $Q$ with $2Q=W_3$, namely,
$$(\lambda, \mp (\lambda+1)\lambda),\  (-\lambda, \pm (\lambda-1)\lambda).$$
This implies that the group $\mathcal{E}_{1,\lambda}(K)$ contains the subgroup generated  by any
%order 4 point
$Q$ and  $W_1$,  which is $\mathbb{Z}/4\mathbb{Z}\oplus \mathbb{Z}/2\mathbb{Z}$.

\begin{rem}
\label{W3L}
Our computations show that if $ {Q}$ is a $K$-point on $ {E}_{1,\lambda}$,
then
$$2{Q}=W_3 \text{ if and only if }  x({Q})=\pm \lambda.$$
Both cases (signs) do occur.
\end{rem}

\begin{rem}
There is another family of elliptic curves  (\cite[Table 3 on p.  217]{Kubert} (see also \cite[Part 2]{Silver}, \cite[Appendix E]{Robledo}))
$$\mathfrak{E}_{1,t}: y^2+xy-\left(t^2-\frac{1}{16}\right)y=x^3-\left(t^2-\frac{1}{16}\right)x^2$$
 whose group of $K$-points contains a subgroup isomorphic to $\mathbb{Z}/4\mathbb{Z}\oplus \mathbb{Z}/2\mathbb{Z}$.
 If we put
 $$y_1:=y+\frac{x-(t^2-\frac{1}{16})}{2},$$
 then the equation  may be rewritten  as
 $$y_1^2=x^3-\left(t^2-\frac{1}{16}\right)x^2+\left[\frac{x-(t^2-\frac{1}{16})}{2}\right]^2=\left(x-t^2+\frac{1}{16}\right)\left(x+\frac{t}{2}+\frac{1}{8}\right)\left(x-\frac{t}{2}+\frac{1}{8}\right).$$
 If we put $x_1:=x-t^2+1/16$, then the equation becomes
 $$y_1^2=x_1\left(x_1+\left(t+\frac{1}{4}\right)^2\right)\left(x_1+\left(t-\frac{1}{4}\right)^2\right),$$
 which defines the elliptic curve
   $\mathcal{E}_{1,\lambda}(1/\kappa)$ with
   $$\lambda=\frac{t-\frac{1}{4}}{t+\frac{1}{4}}, \  \kappa=t+\frac{1}{4}.$$ In particular, $\mathfrak{E}_{1,t}$ is isomorphic to  $\mathcal{E}_{1,\lambda}$.
\end{rem}

\begin{thm}
\label{family2}
Let $E$ be an elliptic curve over $K$.
 Then $E(K)$ contains a subgroup  isomorphic to  $\mathbb{Z}/4\mathbb{Z}\oplus \mathbb{Z}/2\mathbb{Z}$
if and only if
% $K$ contains $\sqrt{-1}$ and
 there exists  $\lambda \in K
\setminus \{0,\pm 1\}$ such that $E$ is isomorphic to $\mathcal{E}_{1,\lambda}$.
\end{thm}

\begin{proof}
We already know that $\mathcal{E}_{1,\lambda}(K)$ contains a subgroup isomorphic to $\mathbb{Z}/4\mathbb{Z}\oplus \mathbb{Z}/2\mathbb{Z}$.
  Conversely,  suppose that $E$ is an elliptic curve over $K$ such that
 $E(K)$ contains a subgroup isomorphic to  $\mathbb{Z}/4\mathbb{Z}\oplus \mathbb{Z}/2\mathbb{Z}$.
Then  $E(K)$ contains all three points of order 2, and  therefore $E$ can be represented in the  form \eqref{E2}.
 It is also clear that at least one of the points \eqref{W2} is divisible by 2 in $E(K)$. Suppose that $W_3$ is divisible by $2$. We may assume that
 $\alpha_3=0$. By Theorem \ref{th0},  both nonzero differences
$$-\alpha_1=\alpha_3-\alpha_1, \  -\alpha_2=\alpha_3-\alpha_2$$
are squares in  $K$; in addition,  they are {\sl distinct} elements of $K$. Thus there are nonzero  $a,b \in K$ such that $a \ne \pm b$ and
$-\alpha_1=a^2,\  -\alpha_2=b^2$.
Since $\alpha_3=0$,   the equation for $E$ is
$$E: y^2=(x+a^2)(x+b^2)x.$$
If we put $\kappa=b$, then we obtain that $E$ is isomorphic to
$$E(\kappa): {y^{\prime}}^2=\left(x^{\prime}+\frac{a^2}{b^2}\right)(x^{\prime}+1) x^{\prime},$$
which is nothing else but $\mathcal{E}_{1,\lambda}$ with
$\lambda=a/b$.
\end{proof}

\begin{cor}
Let $E$ be an elliptic curve over $\mathbb{F}_5$.
The group $E(\mathbb{F}_5)$ is isomorphic to $\mathbb{Z}/4\mathbb{Z}\oplus \mathbb{Z}/2\mathbb{Z}$ if and only if
$E$ is isomorphic to the elliptic curve
$y^2=x^3-x$.
\end{cor}

\begin{proof}
Suppose that $E(\mathbb{F}_5)$ is isomorphic to $\mathbb{Z}/4\mathbb{Z}\oplus \mathbb{Z}/2\mathbb{Z}$.
By Theorem \ref{family2}, $E$ is isomorphic to
$$y^2=(x+\lambda^2)(x+1)x \ \text{ with }
\lambda \in \mathbb{F}_5 \setminus  \{0,1,-1\}.$$
 This implies that $\lambda=\pm 2, \lambda^2=-1$, and so $E$ is isomorphic to
$$\mathcal{E}_{1,2}: y^2=(x-1)(x+1)=x^3-x.$$
Now we need to check that $\mathcal{E}_{1,2}(\mathbb F_5)\cong
\mathbb{Z}/4\mathbb{Z}\oplus \mathbb{Z}/2\mathbb{Z}$.
By Theorem \ref{family2}, $E(\mathbb{F}_5)$ contains a subgroup isomorphic to $\mathbb{Z}/4\mathbb{Z}\oplus \mathbb{Z}/2\mathbb{Z}$; in particular, $8$ divides $|E(\mathbb{F}_5)|$.
In order to finish the proof, it suffices to check that $|E(\mathbb{F}_5)|<16$,
but this inequality follows from the Hasse bound \eqref{HasseB}
$$|E(\mathbb{F}_5)|\le 5+2\sqrt{5}+1<11.$$
\end{proof}

%%%%%%%%%%%%%%%%%%%%changes made on August 5, 2016%%%%%%%%%
\begin{cor}
Let   $E$ be an elliptic curve over $\mathbb{F}_7$.
The group $E(\mathbb{F}_7)$ is isomorphic to $\mathbb{Z}/4\mathbb{Z}\oplus \mathbb{Z}/2\mathbb{Z}$ if and only if
$E$ is isomorphic to
% one of the two
the elliptic curve
$y^2=(x+2)(x+1)x$.
%, \ y^2=(x+4)(x+1)x.$$
\end{cor}

\begin{proof}
Suppose that $E(\mathbb{F}_7)$ is isomorphic to $\mathbb{Z}/4\mathbb{Z}\oplus \mathbb{Z}/2\mathbb{Z}$.
It follows from Theorem \ref{family2} that $E$ is isomorphic to
$y^2=(x+\lambda^2)(x+1)x$
with $\lambda \in \mathbb{F}_7 \setminus \{0,1,-1\}$. This implies that $\lambda=\pm 2$ or $\pm 3$, and therefore $\lambda^2=4$ or $2$, i.e.,  $E$ is isomorphic to
one of the two elliptic curves
$$\mathcal{E}_{1,3}:y^2=(x+2)(x+1)x, \ \mathcal{E}_{1,2}: y^2=(x+4)(x+1)x.$$
Since $ 1/4=2$ in $\mathbb{F}_7$, the elliptic curve  $\mathcal{E}_{1,3}$
coincides with $\mathcal{E}_{1,2}(2)$; in particular, $\mathcal{E}_{1,2}$
 and $\mathcal{E}_{1,3}$ are isomorphic.

Now suppose that $E=\mathcal{E}_{1,2}$.
%%%%%%%%%%%%%%%%%%%%%%%%%%%%
We need to prove that
$E(\mathbb{F}_7)$ is isomorphic to $\mathbb{Z}/4\mathbb{Z}\oplus \mathbb{Z}/2\mathbb{Z}$.  By Theorem \ref{family2}, $E(\mathbb{F}_7)$ contains a subgroup isomorphic to $\mathbb{Z}/4\mathbb{Z}\oplus \mathbb{Z}/2\mathbb{Z}$; in particular, $8$ divides $|E(\mathbb{F}_7)|$. In order to finish the proof, it suffices to check that $|E(\mathbb{F}_7)|<16$, but this inequality follows from the Hasse bound \eqref{HasseB}
$$|E(\mathbb{F}_7)|\le 7+2\sqrt{7}+1<14.$$
\end{proof}

\begin{thm}
\label{full4}
Suppose that $K$ contains $\mathbf{i}=\sqrt{-1}$. Let $a,b$ be nonzero elements of $K$ such that
$a \ne \pm b, \  a \ne \pm \mathbf{i}b$.
Let us consider the elliptic curve
$$E_{a,b}: y^2=(x-\alpha_1)(x-\alpha_2)(x-\alpha_3)$$
%$$E: y^2=\left[x-(a^2-c^2)^2\right]\left[x-(a^2+c^2)^2\right]x$$
 over $K$ with
$\alpha_1=(a^2-b^2)^2, \ \alpha_2=(a^2+b^2)^2,\  \alpha_3=0$.
Then all points of order $2$ on $E$ are divisible by $2$ in $E(K)$, i.e., $E(K)$ contains all twelve points of order $4$. In particular, $E_{a,b}(K)$ contains a subgroup isomorphic to $\mathbb{Z}/4\mathbb{Z}\oplus \mathbb{Z}/4\mathbb{Z}$.
\end{thm}

\begin{proof}
Clearly,  all $\alpha_i$ and $-\alpha_j$ are squares in $K$. In addition,
$$\alpha_2-\alpha_1=(a^2+b^2)^2-(a^2-b^2)^2=(2ab)^2, \
\alpha_1-\alpha_2= (2\mathbf{i}ab)^2. $$
This implies that all $\alpha_i-\alpha_j$ are squares in $K$.  It follows from Theorem \ref{th0} that all
points $W_i=(\alpha_i,0)$  of order $2$ are divisible by $2$ in $E(K)$, and therefore $E(K)$ contains all twelve ($3\times 4$) points of order $4$.
\end{proof}

Keeping the notation and assumptions of Theorem \ref{full4}, we describe explicitly all twelve points of order 4, using formula (\ref{halfP}).

\begin{enumerate}
\item
Dividing the point
$W_2=(\alpha_2,0)=\left((a^2+b^2)^2,0\right)$ by 2, we have $r_2=0$ and get  four choices
$r_1= \pm 2ab, \ r_3=\pm (a^2+b^2)$.
This gives us four points $Q$ with $2Q=W_2$, namely,  two points
$$\left((a^2+b^2)^2+2ab(a^2+b^2),\ \pm (a^2+b^2+2ab)2ab(a^2+b^2)\right)$$
$$=\left((a^2+b^2)(a+b)^2,\ \pm 2ab(a^2+b^2)(a+b)^2\right)$$
and  two points
$\left((a^2+b^2)(a-b)^2,\ \pm  2ab(a^2+b^2)(a-b)^2\right)$.
\item
Dividing the   point  $W_3=(\alpha_3,0)=(0,0)$ by $2$,  we have $r_3=0$ and get  four choices
$r_1=\pm \mathbf{i} (a^2-b^2), \ r_2= \pm \mathbf{i} (a^2+b^2)$.
This gives us four points $Q$ with $2Q=W_3$, namely,  two points
$$\left((a^2-b^2)(a^2+b^2),\ \pm (\mathbf{i}( (a^2-b^2)+\mathbf{i} (a^2+b^2))(a^2-b^2)(a^2+b^2)\right)$$
$$=\left(a^4-b^4,\ \pm 2\mathbf{i}a^2(a^4-b^4)\right)$$
and  two points
$\left(b^4-a^4,\ \pm 2\mathbf{i}b^2(b^4-a^4)\right)$.
\item
Dividing the  point  $W_1=(\alpha_1,0)=\left((a^2-b^2)^2,0\right)$ by $2$,  we have $r_1=0$ and get  four choices
$r_2=\pm 2\mathbf{i}ab, \ r_3=\pm (a^2-b^2)$.
This gives us four points $Q$ with $2Q=W_3$, namely,  two points
$$\left((a^2-b^2)^2+2\mathbf{i}ab(a^2-b^2),\ \pm (2\mathbf{i}ab+(a^2-b^2))2\mathbf{i}ab(a^2-b^2)\right) $$
$$=\left((a^2-b^2)(a+\mathbf{i}b)^2,\ \pm 2\mathbf{i}ab(a^2-b^2)(a+\mathbf{i}b)^2\right)$$
and  two points
$\left((a^2-b^2)(a-\mathbf{i}b)^2, \ \pm 2\mathbf{i}ab(a^2-b^2)(a-\mathbf{i}b)^2\right)$.
\end{enumerate}

\begin{rem}
\label{rem44}
Let $\lambda$ be an element of $K \setminus \{0, \pm 1, \pm \sqrt{-1}\}$. We write $\mathcal{E}_{2,\lambda}$ for the elliptic curve
$$\mathcal{E}_{2,\lambda}: y^2= \left(x+\frac{(\lambda^2-1)^2}{(\lambda^2+1)^2}\right)(x+1)x $$
over $K$.  The elliptic curves $\mathcal{E}_{2,\lambda}$ and $E_{a,b}$ are isomorphic if $a=\lambda b$. Indeed,  one has only to put $\kappa=a^2+b^2$ and notice that
$E_{a,b}(\kappa)=\mathcal{E}_{2,\lambda}$.
It follows from Theorem \ref{full4} that $\mathcal{E}_{2,\lambda}(K)$ contains a subgroup isomorphic to  $\mathbb{Z}/4\mathbb{Z}\oplus \mathbb{Z}/4\mathbb{Z}$.

There is another family of elliptic curves with this property, namely,
$$y^2=x(x-1)\left(x-\frac{(u+u^{-1})^2}{4}\right) $$
(\cite{Shioda}, \cite[pp. 451--453]{Silver};  see also Remark  \ref{equivfam}).
\end{rem}

\begin{thm}
\label{family4}
Let $E$ be an elliptic curve over $K$.
Then $E(K)$ contains a subgroup isomorphic to  $\mathbb{Z}/4\mathbb{Z}\oplus \mathbb{Z}/4\mathbb{Z}$
if and only if $K$ contains $\sqrt{-1}$ and
there exists
\newline
  $\lambda \in K
\setminus \{0,\pm 1, \pm \sqrt{-1} \}$ such that $E$ is isomorphic to $\mathcal{E}_{2,\lambda}$.
\end{thm}

\begin{proof}
Recall (Remark \ref{rem44}) that
$\mathcal{E}_{2,\lambda}(K)$ contains a subgroup isomorphic to
$\mathbb{Z}/4\mathbb{Z}\oplus \mathbb{Z}/4\mathbb{Z}$.

Conversely, suppose that $E$ is an elliptic curve over $K$ and $E(K)$ contains a subgroup isomorphic to
$\mathbb{Z}/4\mathbb{Z}\oplus \mathbb{Z}/4\mathbb{Z}$.
By Theorem \ref{family2}, there is $\delta \in K \setminus \{0,\pm 1\}$ such that $E$ is isomorphic to
$$\mathcal{E}_{1,\delta}:y^2=(x+\delta^2)(x+1)x.$$
Hence we may assume that
$\alpha_1=-\delta^2, \alpha_2=-1, \alpha_3=0$.
It follows from Theorem \ref{th0}  that all
$\pm 1, \pm (\delta^2-1)$
are squares in $K$. (In particular, $\mathbf{i}=\sqrt{-1}$ lies in $K$.) So, there is $\gamma \in K$ with $\gamma^2=1-\delta^2$. Clearly, $\gamma \ne 0, \pm 1$. We have
$$\delta^2+\gamma^2=1.$$
The well-known parametrization of the ``unit circle"  (that goes back to Euler) tells us that there exists $\lambda  \in K$ such that $\lambda^2+1 \ne 0$ and
$$\delta=\frac{\lambda^2-1}{\lambda^2+1}, \ \gamma= \frac{2\lambda}{\lambda^2+1}.$$
Now one has only to plug in  the formula for $\delta$ into the equation of
$\mathcal{E}_{1,\delta}$ and get $\mathcal{E}_{2,\lambda}$.
\end{proof}
\begin{rem}\label{equivfam}
Using a different parametrization of the unit circle in the proof of Theorem \ref{family4}, we obtain the family of elliptic curves
$$E: y^2= \left(x+\frac{(2\lambda)^2}{(\lambda^2+1)^2}\right)(x+1)x$$
with the same property as the family  $\mathcal{E}_{2,\lambda}$. Notice that, for each $\lambda\in K\setminus\{0,\pm1\}$, the elliptic curve $E$ is isomorphic to the elliptic curve $$y^2=x(x-1)\left(x- (u+u^{-1})^2/4 \right)$$ mentioned in Remark \ref{rem44}. Indeed, the latter   differs from $E(\kappa)$, where $ \kappa=2\lambda\sqrt{-1}/(\lambda^2+1)$, only with the change of the parameter $\lambda$ by $u$.
\end{rem}
\begin{cor}
 Let $E$ be an elliptic curve over $\mathbb{F}_q$, where $q=9,13,17$.
The group $E(\mathbb{F}_q)$ is isomorphic to $\mathbb{Z}/4\mathbb{Z}\oplus \mathbb{Z}/4\mathbb{Z}$ if and only if
$E$ is isomorphic to one of    elliptic curves
$\mathcal{E}_{2,\lambda}$.
%%%%%%%%%%%%%%new stuff added on August 5, 2016%%%%%%%%%%%%%
If $q=9$, then $E(\mathbb{F}_q)$ is isomorphic to $\mathbb{Z}/4\mathbb{Z}\oplus \mathbb{Z}/4\mathbb{Z}$ if and only if
$E$ is isomorphic to $y^2=x^3-x$.
%%%%%%%%%%%%%%%%%%%%%%%%%%%%%%%%%%%%%%%%%
\end{cor}

\begin{proof}
First, $\mathbb{F}_q$ contains $\sqrt{-1}$.
Suppose that $E(\mathbb{F}_q)$ is isomorphic to $\mathbb{Z}/4\mathbb{Z}\oplus \mathbb{Z}/4\mathbb{Z}$.
It follows from Theorem \ref{family4} that $E$ is isomorphic to
$\mathcal{E}_{2,\lambda}$.
%: y^2=(x+\lambda^2)(x+1)x$$
%with  $\lambda\in \mathbb{F}_q\setminus \{0,1, \pm \sqrt{-1}\}$.

Conversely, suppose that $E$ is isomorphic to  one of these curves. We need to prove that
$E(\mathbb{F}_q)$ is isomorphic to $\mathbb{Z}/4\mathbb{Z}\oplus \mathbb{Z}/4\mathbb{Z}$.  By Theorem \ref{family4}, $E(\mathbb{F}_q)$ contains a subgroup isomorphic to $\mathbb{Z}/4\mathbb{Z}\oplus \mathbb{Z}/4\mathbb{Z}$; in particular, $16$ divides $|E(\mathbb{F}_q)|$. In order to finish the proof, it suffices to check that $|E(\mathbb{F}_q)|<32$, but this inequality follows from the Hasse bound \eqref{HasseB}
$$|E(\mathbb{F}_q)|\le q+2\sqrt{q}+1\le 17+2\sqrt{17}+1<27.$$

Now assume that $q=9$. Then $\lambda$
%new stuff added on Jan. 26, 2017%%%%%
is one of four
 $\pm (1\pm \mathbf{i})$.
 %%%%%%%%%%%%%%%%%%%
  For all such $\lambda$ we have
$$\lambda^2= \pm 2\mathbf{i}=\mp \mathbf{i},  \
\frac{(\lambda^2-1)^2}{(\lambda^2+1)^2}=\frac{(1\mp \mathbf{i})^2}{(-1\mp \mathbf{i})^2}=\frac{\mp 2\mathbf{i}}{\pm 2\mathbf{i}}=-1.$$
Therefore the equation for $\mathcal{E}_{2,\lambda}$ is
$$y^2=(x-1)(x+1)x=x^3-x.$$

\end{proof}

\begin{cor}
 Let $E$ be an elliptic curve over $\mathbb{F}_{29}$.
The group $E(\mathbb{F}_{29})$ is isomorphic to $\mathbb{Z}/8\mathbb{Z}\oplus \mathbb{Z}/4\mathbb{Z}$ if and only if
$E$ is isomorphic to one of   elliptic curves
$\mathcal{E}_{2,\lambda}$.
%: y^2= \left(x+\frac{(\lambda^2-1)^2}{(\lambda^2+1)^2}\right)(x+1)x $$
%with $\lambda\in \mathbb{F}_q\setminus \{0, \pm 1, \pm \sqrt{-1}\}$.
\end{cor}

\begin{proof}
First, $\mathbb{F}_{29}$ contains $\sqrt{-1}$.
Suppose that $E(\mathbb{F}_{29})$ is isomorphic to $\mathbb{Z}/8\mathbb{Z}\oplus \mathbb{Z}/4\mathbb{Z}$.
Then $E(\mathbb{F}_{29})$ contains a subgroup isomorphic to $\mathbb{Z}/4\mathbb{Z}\oplus \mathbb{Z}/4\mathbb{Z}$.
It follows from Theorem \ref{family4} that $E$ is isomorphic to
$\mathcal{E}_{2,\lambda}$.
%:y^2=(x+\lambda^2)(x+1)x$$
%with  $\lambda\in \mathbb{F}_q\setminus \{0,1, \pm \sqrt{-1}\}$.

Conversely, suppose that $E$ is isomorphic to one of these curves. We need to prove that
$E(\mathbb{F}_{29})$ is isomorphic to $\mathbb{Z}/8\mathbb{Z}\oplus \mathbb{Z}/4\mathbb{Z}$.  By Theorem \ref{family4}, $E(\mathbb{F}_{29})$ contains a subgroup isomorphic to $\mathbb{Z}/4\mathbb{Z}\oplus \mathbb{Z}/4\mathbb{Z}$; in particular, $16$ divides $|E(\mathbb{F}_{29})|$. The  Hasse bound \eqref{HasseB} tells us that
$$29+1-2\sqrt{29} \le |E(\mathbb{F}_q)| \le 29+1+2\sqrt{29}$$ and therefore
$$19<  |E(\mathbb{F}_{29})|< 41.$$
It follows that $|E(\mathbb{F}_{29})|=32$; in particular,  $E(\mathbb{F}_{29})$ is a finite $2$-group. Clearly, $E(\mathbb{F}_{29})$ is isomorphic to a product of two cyclic $2$-groups, each of which has order divisible by $4$.  It follows that $E(\mathbb{F}_{29})$ is isomorphic to $\mathbb{Z}/8\mathbb{Z}\oplus \mathbb{Z}/4\mathbb{Z}$.
\end{proof}
%%%%%%%%%%%%%%%%%%%%Added Jan. 15, 2017%%%%%%%%%%
\begin{thm}
\label{Qi}
Let $K=\Q(\sqrt{-1})$ and
$E$ be an elliptic curve over $\Q(\sqrt{-1})$.
Then the torsion subgroup $E(\Q(\sqrt{-1})_t$ of  $E(\Q(\sqrt{-1}))$ is isomorphic to  $\mathbb{Z}/4\mathbb{Z}\oplus \mathbb{Z}/4\mathbb{Z}$
if and only if
there exists
  $\lambda \in K
\setminus \{0,\pm 1, \pm \sqrt{-1} \}$ such that $E$ is isomorphic to $\mathcal{E}_{2,\lambda}$.
\end{thm}

\begin{proof}
By Theorem \ref{Kquad},
if $E(\Q(\sqrt{-1}))$ contains a subgroup  isomorphic to  $\mathbb{Z}/4\mathbb{Z}\oplus \mathbb{Z}/4\mathbb{Z}$ then
$E(\Q(\sqrt{-1})_t$  is isomorphic to  $\mathbb{Z}/4\mathbb{Z}\oplus \mathbb{Z}/4\mathbb{Z}$. Now the desired result follows from Theorem \ref{family2}.
\end{proof}

\section{Points of order 8}
\label{l8}
Let us return to the curve $\mathcal{E}_{1,\lambda}$ and consider ${Q}\in \mathcal{E}_{1,\lambda}(K)$ with $2{Q}=W_3$.
 Let us try to divide ${Q}$ by $2$ in $E(K)$.
By Remark \ref{W3L},  $x({Q})=\pm \lambda$.
First, we assume that  $x( {Q})= \lambda$  (such a ${Q}$ does exist).

\begin{lem}
\label{divLambda}
Let $Q$ be a point of  $\mathcal{E}_{1,\lambda}(K)$ with $x( {Q})= \lambda$. Then $Q$ is divisible by $2$ in  $\mathcal{E}_{1,\lambda}(K)$ if and only if there exists
$c \in K \setminus \{ 0, \pm1, \pm 1\pm \sqrt{2},  \pm \sqrt{-1}\}$
such that
$$\lambda=\left[\frac{c-\frac{1}{c}}{2}\right]^2.$$
\end{lem}

\begin{proof}
We have
$$\lambda-\alpha_1=\lambda-(-\lambda^2)=\lambda+\lambda^2, \ \lambda-\alpha_2=\lambda-(-1)=\lambda+1,\ \lambda-\alpha_3=\lambda-0=\lambda.$$
 By Theorem \ref{th0}, $ {Q} \in 2\mathcal{E}_{1,\lambda}(K)$  if and only if all three
 $\lambda+\lambda^2, \lambda+1, \lambda$ are squares in $K$. The latter means that both $\lambda$ and $\lambda+1$ are squares in $K$, i.e.,  there exist $a,b\in K$ such that
$a^2=\lambda+1, \lambda=b^2$.
This implies that
the pair $(a, b)$ is a $K$-point on the hyperbola
$$u^2-v^2=1.$$
Recall that $\lambda \ne 0,\pm 1$.
Using the well-known parametrization
$$u=\frac{t+\frac{1}{t}}{2}, v=\frac{t-\frac{1}{t}}{2}$$
of the hyperbola, we obtain that
 both $\lambda$ and $\lambda+1$ are squares in $K$
if and only if
 there exists a {\sl nonzero} $c \in K$ such that
\begin{equation*}
\label{lambda8plus}
\lambda=\left[\frac{c-\frac{1}{c}}{2}\right]^2.
\end{equation*}
If this is the case, then
$$a=\pm \frac{c+\frac{1}{c}}{2},\  b=\pm \frac{c-\frac{1}{c}}{2}$$
and
$$\lambda+1=\left[\frac{c+\frac{1}{c}}{2}\right]^2.$$
Recall that $\lambda \ne 0, \pm 1$. This means that
$$\frac{c-\frac{1}{c}}{2} \ne 0, \pm 1, \pm \sqrt{-1}, \ \text{ i.e., }$$
\begin{equation*}
c \ne 0, \pm1, \pm 1\pm \sqrt{2},  \pm \sqrt{-1}.
\end{equation*}
\end{proof}

Now let us assume that $x({Q})=-\lambda$ (such a $ {Q}$ does exist).

\begin{lem}
\label{divMLambda}
Let $Q$ be a point of  $\mathcal{E}_{1,\lambda}(K)$ with $x( {Q})=- \lambda$. Then $Q$ is divisible by $2$ in  $\mathcal{E}_{1,\lambda}(K)$ if and only if there exists
$c \in K \setminus \{ 0, \pm1, \pm 1\pm \sqrt{2},  \pm \sqrt{-1}\}$
 such that
$$\lambda=-\left[\frac{c-\frac{1}{c}}{2}\right]^2.$$
\end{lem}

\begin{proof}
Applying Lemma  \ref{divLambda} to $-\lambda$ (instead of $\lambda$) and the curve
$\mathcal{E}_{1,-\lambda}=\mathcal{E}_{1,\lambda}$, we obtain that
$ {Q}\in 2\mathcal{E}_{1,-\lambda}(K)=2\mathcal{E}_{1,\lambda}(K)$
if and only if there exists
$$c \in K \setminus \{ 0, \pm1, \pm 1\pm \sqrt{2},  \pm \sqrt{-1}\}$$
such that
\begin{equation*}
\label{lambda8minus}
-\lambda=\left[\frac{c-\frac{1}{c}}{2}\right]^2.
\end{equation*}
\end{proof}

Lemmas \ref{divLambda} and \ref{divMLambda}
%Formulas (\ref{lambda8plus}) and (\ref{lambda8minus})
give us the following statement.

\begin{prop}
\label{div4W3}
The point $W_3=(0,0)$ is divisible by $4$ in $\mathcal{E}_{1,\lambda}(K)$ if and only if
there exists  $c \in K$ such that
 $c \ne  0, \pm 1, \pm 1\pm \sqrt{2},  \pm \sqrt{-1}$ and
$$\lambda=\pm \left[\frac{c-\frac{1}{c}}{2}\right]^2, \
\text{ i.e., } \
\lambda^2= \left[\frac{c-\frac{1}{c}}{2}\right]^4.$$
\end{prop}

%%%%%%%%%%%new stuff added on July 18, 2016%%%%%%%%%%
\begin{prop}
\label{W3by4}
The following conditions are equivalent.

\begin{itemize}
\item[(i)]
If $Q \in \mathcal{E}_{1,\lambda}(K)$  is any point with $2Q=W_3$, then  $Q$  lies in $2 \mathcal{E}_{1,\lambda}(K)$.
\item[(ii)]
If $R$ is any point of $\mathcal{E}_{1,\lambda}$ with $4R=W_3$, then $R$ lies in $\mathcal{E}_{1,\lambda}(K)$.
\item[(iii)]
 There exist
$c,d \in K\setminus \{ 0, \pm1, \pm 1\pm \sqrt{2},  \pm \sqrt{-1}\}$
such that
$$\lambda= \left[\frac{c-\frac{1}{c}}{2}\right]^2, \ -\lambda= \left[\frac{d-\frac{1}{d}}{2}\right]^2.$$
\end{itemize}
If these equivalent conditions hold, then $K$ contains $\sqrt{-1}$ and $\mathcal{E}_{1,\lambda}(K)$ contains all (twelve) points of order 4.
\end{prop}

\begin{proof}
The equivalence of (i) and (ii) is obvious. It is also clear that (ii) implies that all points of order (dividing) 4 lie in $\mathcal{E}_{1,\lambda}(K)$.

Recall (Remark \ref{W3L}) that $Q$ with $2Q=W_3$ are exactly the points of  $\mathcal{E}_{1,\lambda}$ with $x(Q)=\pm \lambda$. Now the equivalence of (ii) and (iii) follows from Lemmas \ref{divLambda} and \ref{divMLambda}.

In order to finish the proof, we notice that $\lambda \ne 0$ and
$$-1=\frac{-\lambda}{\lambda}=\left[\frac{ \left[\frac{d-\frac{1}{d}}{2}\right]}{\left[\frac{c-\frac{1}{c}}{2}\right]}\right]^2.$$

\end{proof}

%%%%%%%%%%%%%%%%%%%%%%%%%%%%%%%%%%%%%

 Suppose that
$$\lambda= \left[\frac{c-\frac{1}{c}}{2}\right]^2 \ \text{ with }
 c \in K \setminus \{  0, \pm1, \pm 1\pm \sqrt{2},  \pm \sqrt{-1}\}$$
and  consider
$Q=(\lambda, (\lambda+1)\lambda)\in \mathcal{E}_{1,\lambda}(K)$  of order $4$ with  $2Q=W_3$.
Let us find a point $R \in \mathcal{E}_{1,\lambda}(K)$ of order $8$ with $2R=Q$.  First, notice that
$$Q=(\lambda, (\lambda+1)\lambda)=
\left(\left[\frac{c-\frac{1}{c}}{2}\right]^2, \left[\frac{c+\frac{1}{c}}{2}\right]^2 \cdot \left[\frac{c-\frac{1}{c}}{2}\right]^2 \right)$$
$$=\left(\frac{(c^2-1)^2}{4c^2},\frac{(c^4-1)^2}{4c^4}\right).$$
We have
$$r_1=\sqrt{\lambda+\lambda^2}=\sqrt{(\lambda+1)\lambda}, \ r_2=\sqrt{\lambda+1}, \ r_3=\sqrt{\lambda}; \ r_1 r_2 r_3=-(\lambda+1)\lambda.$$
This means that
$$r_1=\pm \frac{c-\frac{1}{c}}{2}\cdot \frac{c+\frac{1}{c}}{2}, \  r_2=\pm \frac{c+\frac{1}{c}}{2}, \ r_3= \pm  \frac{c-\frac{1}{c}}{2},$$
and the signs should be chosen in such a way that the product $\ r_1 r_2 r_3$ coincides with
$$-\left[\frac{c-\frac{1}{c}}{2}\right]^2\cdot \left[\frac{c+\frac{1}{c}}{2}\right]^2.$$
For example, we may take
$$r_1=- \frac{c-\frac{1}{c}}{2}\cdot \frac{c+\frac{1}{c}}{2}=-\frac{c^2-\frac{1}{c^2}}{4}=-\frac{c^4-1}{4c^2}, \  r_2= \frac{c+\frac{1}{c}}{2}, \ r_3=   \frac{c-\frac{1}{c}}{2}$$
and get (since $r_2+r_3=c$ and $r_2 r_3=(c^4-1)/4c^2)$)
$$r_1+r_2+r_3=-\frac{c^4-1}{4c^2}+c=\frac{-c^4+4c^3+1}{4c^2},$$
$$r_1 r_2+r_2 r_3+r_3 r_1=c r_1+r_2 r_3= -\frac{c(c^4-1)}{4c^2}+
\frac{c^4-1}{4c^2}
 =\frac{(1-c)(c^4-1)}{4c^2}.$$
Now  \eqref{x1} and  \eqref{chap}  tell us that the coordinates of the corresponding $R$ with $2R=Q$ are as follows:
$$x(R)=x(Q)+r_1 r_2+r_2 r_3+r_3 r_1=
\frac{(c^2-1)^2}{4c^2}+\frac{(1-c)(c^4-1)}{4c^2}=\frac{(1-c)^3(c+1)}{4c},$$
$$y(R)=-(r_1+r_2)(r_2+r_3)(r_1+r_3)=$$
$$-\left(- \frac{c-\frac{1}{c}}{2}\cdot \frac{c+\frac{1}{c}}{2}+ \frac{c+\frac{1}{c}}{2}\right) c
\left(- \frac{c-\frac{1}{c}}{2}\cdot \frac{c+\frac{1}{c}}{2}+ \frac{c-\frac{1}{c}}{2}\right) =$$
$$-\left(1- \frac{c-\frac{1}{c}}{2}\right)\cdot \frac{c+\frac{1}{c}}{2}\cdot
 c \cdot \left(1- \frac{c+\frac{1}{c}}{2}\right) \frac{c-\frac{1}{c}}{2}=$$
 $$-\frac{c^2-\frac{1}{c^2}}{16}\cdot \left(c-2-\frac{1}{c}\right)\left(c-2+\frac{1}{c}\right) c=
-\frac{\left(c^2-\frac{1}{c^2}\right)\left((c-2)^2-\frac{1}{c^2} \right)c}{16}.$$
So, we get the $K$-point of order 8
\begin{equation*}
\label{order8}
R=\left(\frac{(1-c)^3(c+1)}{4c}, -\frac{\left(c^2-\frac{1}{c^2}\right)\left((c-2)^2-\frac{1}{c^2}\right)c}{16}\right)
\end{equation*}
on the elliptic curve
$$\mathcal{E}_{4,c}:=\mathcal{E}_{1,\left(\pm \frac{c-\frac{1}{c}}{2}\right)^2}:
y^2=\left[x+\left(\frac{c-\frac{1}{c}}{2}\right)^4\right](x+1)x$$
for any $c \in K\setminus \{0,\pm 1,\pm 1 \pm \sqrt{2}, \pm \sqrt{-1}\}$.
The group $\mathcal{E}_{4,c}(K)$ contains the subgroup generated by $R$ and $W_1$, which is isomorphic to $\mathbb{Z}/8\mathbb{Z}\oplus \mathbb{Z}/2\mathbb{Z}$.

\begin{thm}
\label{family8}
Let $E$ be an elliptic curve over $K$.
 Then $E(K)$ contains a subgroup  isomorphic to  $\mathbb{Z}/8\mathbb{Z}\oplus \mathbb{Z}/2\mathbb{Z}$
if and only if
 there exists  $c \in K
\setminus \{ 0, \pm1, \pm 1\pm \sqrt{2},  \pm \sqrt{-1}\}$ such that $E$ is isomorphic to $\mathcal{E}_{4,c}$.
\end{thm}

\begin{proof}
We know that  $\mathcal{E}_{4,c}(K)$ contains a subgroup isomorphic to $\mathbb{Z}/8\mathbb{Z}\oplus \mathbb{Z}/2\mathbb{Z}$.

Conversely, suppose that $E(K)$ contains a subgroup  isomorphic to  $\mathbb{Z}/8\mathbb{Z}\oplus \mathbb{Z}/2\mathbb{Z}$. This implies that $E(K)$ contains all three points of order 2, i.e.,
$E$ may be represented in the form \eqref{E2}.
Clearly, one of the points \eqref{W2} is divisible by 4
in $E(K)$. We may assume that $W_3$ is divisible by 4. We may also assume that $\alpha_3=0$, i.e., $W_3=(0,0)$. Then we know that there exist  distinct nonzero $a,b \in K$ such that $\alpha_1=-a^2, \alpha_2=-b^2$, i.e., the equation of $E$ is
$$y^2=(x+a^2)(x+b^2)x.$$
Replacing $E$ by $E(b)$ and putting $\lambda=a/b$, we may  assume that
$$E=\mathcal{E}_{1,\lambda}:y^2=(x+\lambda^2)(x+1)x.$$
Since $W_3$ is divisible by 4 in $\mathcal{E}_{1,\lambda}(K)$, the desired result follows from Proposition \ref{div4W3}.
\end{proof}

\begin{rem}
There is
another family of elliptic curves    (\cite[Table 3 on p.  217]{Kubert}, \cite[Appendix E]{Robledo}))
$$y^2+(1-a(t))xy -b(t)y=x^3-b(t)x^2$$
with
 $$a(t)=\frac{(2t+1)(8t^2+4t+1)}{2(4t+1)(8t^2-1)t},\  b(t)=\frac{(2t+1)(8t^2+4t+1)}{(8t^2-1)^2},$$
 whose group of rational points contains a subgroup isomorphic to $\mathbb{Z}/8\mathbb{Z}\oplus \mathbb{Z}/2\mathbb{Z}$.

 Let us assume that $t$ is an element  of an arbitrary field  $K$ (with $\fchar(K)\ne 2)$ such that
 $$t \ne 0, \ 8t^2-1 \ne 0, \ 4t+1  \ne 0$$
 and  put
 $$U(t):=(2t+1)(8t^2+4t+1), \ A(t)=2(4t+1)(8t^2-1)t \ne 0, \ B(t)= (8t^2-1)^2 \ne 0,$$
 $$ a(t)=\frac{U(t)}{A(t)},  \ b(t)=\frac{U(t)}{B(t)}.$$
 Let us  consider the
 cubic curve $\mathfrak{E}_{4,t}$ over $K$ defined by the same equation
 $$\mathfrak{E}_{4,t}: y^2+(1-a(t))xy -b(t)y=x^3-b(t)x^2$$
  as above.
In light of Theorem \ref{family8}, if $\mathfrak{E}_{4,t}$ is an elliptic curve over $K$,  then $\mathfrak{E}_{4,t}$ is isomorphic to $\mathcal{E}_{4,c}$ for a certain $c \in K$.
 Let us find the corresponding $\lambda$ (as a rational function of $t$).
  First, rewrite the equation for $\mathcal{E}_{4,t}$ as
 $$\left(y+\frac{(1-a(t)x)-b(t)}{2}\right)^2=x^3-b(t)x^2+\left(\frac{(1-a(t))x-b(t)}{2}\right)^2,$$
 i.e.,
 $$\left(y+\frac{(1-a(t)x)-b(t)}{2}\right)^2=x^3-\frac{U(t)}{B(t)}\cdot x^2+\left(\frac{\left(1-\frac{U(t)}{A(t)}\right)x-\frac{U(t)}{B(t)}}{2}\right)^2,$$
 Second, multiplying the last equation by $(A(t)B(t))^6$ and introducing new variables
 $$y_1=(A(t)B(t))^3\cdot \left(y+\frac{(1-a(t))x-b(t)}{2}\right), \ x_1=(A(t)B(t))^2\cdot x,$$
 we obtain (with help of {\bf magma}) the following equation for an isomorphic cubic curve $\tilde{\mathfrak{E}}_{4,t}:$
 $$\begin{aligned} y_1^2=x_1^3+\frac{-U(t)A(t)^2 B(t)+((U(t)-A(t))^2 B(t)^2}{4} x_1^2\\+\frac{(U(t)-A(t)) U(t)A(t)^3 B(t)^3}{2} x_1+\frac{A(t)^6 B(t)^4 U(t)^2}{4}\end{aligned}$$
 $$=(x_1-\alpha_1)(x_1-\alpha_2)(x_1-\alpha_3),$$
where
 $$\begin{aligned} \alpha_1= -(-4194304 t^{15} - 5242880 t^{14} - 262144 t^{13} + 2162688 t^{12} +
 753664 t^{11}\\
        - 262144 t^{10} - 172032 t^9
       - 2048 t^8 + 14336 t^7 + 2304 t^6 -
 320 t^5 -
         112 t^4 - 8t^3), \end{aligned} $$

 $$\begin{aligned}\alpha_2= -(4194304 t^{16} + 4194304 t^{15} - 1048576 t^{14} - 2359296 t^{13} - 327680 t^{12}\\
        + 491520 t^{11} + 163840 t^{10}
       - 40960 t^9 - 25600 t^8  + 1792 t^6 + 192 t^5
        - 48 t^4 - 8 t^3  ,\end{aligned}$$

  $$ \begin{aligned}\alpha_3=-(-4194304 t^{15} - 5242880 t^{14} - 262144 t^{13} + 2424832 t^{12} + 1015808 t^{11}\\
        - 294912 t^{10} - 286720 t^9
      - 25600 t^8 + 30720 t^7 + 8960 t^6 - 832 t^5\\
        - 720 t^4 - 72 t^3 + 16 t^2 + 4 t + 1/4).\end{aligned}$$
  Using {\bf magma}, we obtain that
  $$\alpha_2-\alpha_1=-2^{22} t^4 (t+1/2)^4(t^2-1/8)^4, \ \alpha_3-\alpha_1=-2^{18}(t+1/4)^4 (t^2-1/8)^4.$$
  This implies that $\tilde{\mathfrak{E}}_{4,t}$ (and therefore $\mathfrak{E}_{4,t}$) is an elliptic curve over $K$ (i.e., all three
  $\alpha_1, \alpha_2,\alpha_3$ are {\sl distinct} elements of $K$)
   if and only if
  $$t \ne 0, -\frac{1}{2}, -\frac{1}{4}, \pm \frac{1}{2\sqrt{2}}$$
  and
  $$\frac{\alpha_2-\alpha_1}{\alpha_3-\alpha_1}=\left(\frac{2t(t+1/2)}{t+1/4}\right)^4 \ne 1.$$
  Assume that all these conditions hold. Then the change of variable $x_2=x_1+\alpha_1$ transforms $\tilde{\mathfrak{E}}_{3,t}$
  to the elliptic curve
  $$E: y_1^2=x_2(x_2-(\alpha_2-\alpha_1))(x_2-(\alpha_3-\alpha_1))=$$
  $$x_2\left(x_2+2^{22} t^4 (t+1/2)^4(t^2-1/8)^4\right)\left(x_2+2^{18}(t+1/4)^4 (t^2-1/8)^4\right).$$
  If we put $\kappa=2^9 (t+1/4)^2 (t^2-1/8)^2$, then
  $$\kappa^2=-(\alpha_3-\alpha_1)$$ and
  $E$ is isomorphic to the elliptic curve
  $$E(\kappa): {y^{\prime}}^2=x^{\prime}\left(x^{\prime}+ \frac{\alpha_2-\alpha_1}{\alpha_3-\alpha_1}\right)(x^{\prime}+1)
  =x^{\prime}\left(x^{\prime}+ \left(\frac{2t(t+1/2)}{t+1/4}\right)^4\right)(x^{\prime}+1).$$
  Notice that
  $$\frac{2t(t+1/2)}{t+1/4}= \frac{2t(4t+2)}{(4t+1)} =  \frac{4t(4t+2)}{2(4t+1)} =   \frac{(4t+1)^2-1}{2(4t+1)}=
        \frac{(4t+1)-\frac{1}{(4t+1)}}{2},
        $$
  and therefore $E(\kappa)=\mathcal{E}_{4,c}$ with $c=(4t+1)$. This implies that $\mathfrak{E}_{4,t}$ is isomorphic to
  $\mathcal{E}_{4,c}$ with $c=(4t+1)$.

  \end{rem}

\begin{rem}
Suppose that $K=\mathbb{F}_q$ with $q=3,5,7$ or $9$. Then
$$\mathbb{F}_q
\setminus \{ 0,1,-1, \pm 1\pm \sqrt{2},  \pm \sqrt{-1}\}=\emptyset .$$

\end{rem}

\begin{cor}
Let $E$ be an elliptic curve over $\mathbb{F}_q$, where $q=11,13,17,19$. The group $E(\mathbb{F}_q)$ is isomorphic to $\mathbb{Z}/8\mathbb{Z}\oplus \mathbb{Z}/2\mathbb{Z}$ if and only if
$E$ is isomorphic to one of    elliptic curves
$\mathcal{E}_{4,c}$.
%:y^2=\left(x+\left[\frac{c-\frac{1}{c}}{2}\right]^4\right)(x+1)x$$
%with $c \in K\setminus \{0,1,-1, \pm 1\pm \sqrt{2},  \pm \sqrt{-1}\}$.
\end{cor}

\begin{proof}
Suppose that $E(\mathbb{F}_q)$ is isomorphic to $\mathbb{Z}/8\mathbb{Z}\oplus \mathbb{Z}/2\mathbb{Z}$.
It follows from Theorem \ref{family8} that $E$ is isomorphic to
one of the  elliptic curves
$$\mathcal{E}_{4,c}:
y^2=\left[x+\left(\frac{c-\frac{1}{c}}{2}\right)^4\right](x+1)x$$
with $c \in K\setminus \{0,\pm 1,\pm \sqrt{-1}, \pm \sqrt{-1}\}$.
Conversely, suppose that $E$ is isomorphic to one of these  curves. We need to prove that
$E(\mathbb{F}_q)$ is isomorphic to $\mathbb{Z}/8\mathbb{Z}\oplus \mathbb{Z}/2\mathbb{Z}$.  By Theorem \ref{family8}, $E(\mathbb{F}_q)$ contains a subgroup isomorphic to $\mathbb{Z}/8\mathbb{Z}\oplus \mathbb{Z}/2\mathbb{Z}$; in particular, $16$ divides $|E(\mathbb{F}_q)|$. In order to finish the proof, it suffices to check that $|E(\mathbb{F}_q)|<32$, but this inequality follows from the Hasse bound \eqref{HasseB}
$$|E(\mathbb{F}_q)|\le q+2\sqrt{q}+1\le 19+2\sqrt{19}+1<29.$$
\end{proof}

\begin{cor}
Let $E$ be an elliptic curve over $\mathbb{F}_{47}$.  The group $E(\mathbb{F}_{47})$ is isomorphic to $\mathbb{Z}/24\mathbb{Z}\oplus \mathbb{Z}/2\mathbb{Z}$ if and only if
$E$ is isomorphic to one of   elliptic curves
$\mathcal{E}_{4,c}$.
%y^2=\left(x+\left[\frac{c-\frac{1}{c}}{2}\right]^4\right)(x+1)x$$
%with $c \in K\setminus \{0,1,-1, \pm 1\pm \sqrt{2},  \pm \sqrt{-1}\}$.
\end{cor}

\begin{proof}
Suppose that $E(\mathbb{F}_{47})$ is isomorphic to $\mathbb{Z}/24\mathbb{Z}\oplus \mathbb{Z}/2\mathbb{Z}$. Then it contains a subgroup isomorphic to
 $\mathbb{Z}/8\mathbb{Z}\oplus \mathbb{Z}/2\mathbb{Z}$.
It follows from Theorem \ref{family8} that $E$ is isomorphic to
one of   elliptic curves
$$\mathcal{E}_{4,c}:
y^2=\left[x+\left(\frac{c-\frac{1}{c}}{2}\right)^4\right](x+1)x$$
with  $c \in K\setminus \{0, \pm1, \pm 1\pm \sqrt{2},  \pm \sqrt{-1}\}$.

Conversely, suppose that $E$ is isomorphic to one of these  curves. We need to prove that
$E(\mathbb{F}_{47})$ is isomorphic to $\mathbb{Z}/24\mathbb{Z}\oplus \mathbb{Z}/2\mathbb{Z}$.  By Theorem \ref{family8}, $E(\mathbb{F}_{47})$ contains a subgroup isomorphic to $\mathbb{Z}/8\mathbb{Z}\oplus \mathbb{Z}/2\mathbb{Z}$; in particular, $16$ divides $|E(\mathbb{F}_{47})|$. The Hasse bound tells us that
$$47+1-2\sqrt{47}\le |E(\mathbb{F}_{47})|\le 47+1+2\sqrt{47}$$
and therefore
$34<|E(\mathbb{F}_{47})|<62$.
This implies that $|E(\mathbb{F}_{47})|=48$; in particular, $E(\mathbb{F}_{47})$  contains a point of order 3. This implies that $E(\mathbb{F}_{47})$  contains a subgroup isomorphic to
 $$(\mathbb{Z}/8\mathbb{Z}\oplus \mathbb{Z}/2\mathbb{Z}) \oplus \mathbb{Z}/3\mathbb{Z}\cong \mathbb{Z}/24\mathbb{Z}\oplus \mathbb{Z}/2\mathbb{Z}.$$
 Since this subgroup has the same order 48 as the whole group $E(\mathbb{F}_{47})$, we get the desired result.
\end{proof}

%%%%%%%%%%%%%new stuff added on January 15, 2017%%%%%%%%%%%%%%%%%%%%%
\begin{thm}
\label{Q8}
Let $K=\Q$ and
 $E$ be an elliptic curve over $\Q$.
 Then the torsion subgroup $E(\Q)_t$ of $E(\Q)$ is  isomorphic to  $\mathbb{Z}/8\mathbb{Z}\oplus \mathbb{Z}/2\mathbb{Z}$
if and only if
 there exists  $c \in \Q
\setminus \{ 0, \pm 1\}$
%, \pm 1\pm \sqrt{2},  \pm \sqrt{-1}\}$
 such that $E$ is isomorphic to $\mathcal{E}_{4,c}$.
\end{thm}

\begin{proof}
By Theorem \ref{mazurQ} applied to $m=4$, if $E(\Q)$ contains a subgroup isomorphic
to  $\mathbb{Z}/8\mathbb{Z}\oplus \mathbb{Z}/2\mathbb{Z}$ then $E(\Q)_t$ is isomorphic
to  $\mathbb{Z}/8\mathbb{Z}\oplus \mathbb{Z}/2\mathbb{Z}$. Now the desired result follows from Theorem \ref{family8}, since
neither $\sqrt{2}$ nor $\sqrt{-1}$ lie in $\Q$.
\end{proof}

%%%%%%%%%%%%%%%%%%%%%%%%%%%%%%%%%%%%%%%%%%%%%
\begin{thm}
\label{family84}
Let $E$ be an elliptic curve over $K$.
 Then $E(K)$ contains a subgroup  isomorphic to  $\mathbb{Z}/8\mathbb{Z}\oplus \mathbb{Z}/4\mathbb{Z}$
if and only if $K$ contains $\mathbf{i}=\sqrt{-1}$ and
 there exist  $$c,d \in K
\setminus \{ 0, \pm1, \pm 1\pm \sqrt{2},  \pm \sqrt{-1}\} \ \text{ such that } \
c-\frac{1}{c}=\mathbf{i}\left(d-\frac{1}{d}\right)$$
and
$E$ is isomorphic to $\mathcal{E}_{4,c}$.
\end{thm}
\begin{rem} The above equation defines an open dense set in the plane affine curve \begin{equation}
\label{cd}
\mathcal{M}_{8,4}:(c^2-1)d=\mathbf{i}(d^2-1)c.
\end{equation}
It is immediate that the corresponding projective closure is a nonsingular cubic $\bar{\mathcal{M}}_{8,4}$ with a $K$-point, i.e., an elliptic curve. To obtain a Weierstrass normal form of $\bar{\mathcal{M}}_{8,4}$, we first slightly simplify equation\eqref{cd} by the change of variables
$d=s, \mathbf{i}c=t$ and get $s^2t+ts^2+s-t=0$. Then, using the birational transformation
$$s=\frac{\eta}{\xi+\xi^2},\ t=\frac \eta{1+\xi},$$ we obtain $\eta^2=\xi^3-\xi$.
%%%%%%%%Added Jan. 27%%%%%%%%%%%%%%
\footnote{See \cite[Example 1.4.2 on p. 88]{Silver2} for an explicit description of
 the (finite) set of all  $\Q(\mathbf{i})$-points on this elliptic curve; none of them
corresponds to $(c,d)$ that satisfy the conditions of Theorem \ref{family84}.}
%%%%%%%%%%%%%%%%%%%%%%%%%%%%%%%%%%%%%%

\end{rem}
\begin{proof}[Proof of Theorem \ref{family84}]
We have already seen that  $\mathcal{E}_{4,c}(K)$ contains an order 8 point $R$ with $4R=W_3$.  It follows from Proposition \ref{W3by4} that  $\mathcal{E}_{4,c}(K)$ contains all points of order $4$. In particular, it contains an order 4 point
$\mathcal{Q}$ with $2\mathcal{Q}=W_1$. Clearly, $R$ and $\mathcal{Q}$ generate a subgroup isomorphic to $\mathbb{Z}/8\mathbb{Z}\oplus \mathbb{Z}/4\mathbb{Z}$.

Conversely, suppose that $E(K)$ contains a subgroup  isomorphic to  $\mathbb{Z}/8\mathbb{Z}\oplus \mathbb{Z}/4\mathbb{Z}$. This implies that $E(K)$ contains  all twelve points of order 4. In particular,
$E$ may be represented in the form \eqref{E2}.
Clearly, one of the points of order 2 is divisible by 4
in $E(K)$. We may assume that $W_3$ is divisible by 4.
%We may also assume that $\alpha_3=0$, i.e., $W_3=(0,0)$. Then we know that there exist  distinct nonzero $a,b \in K$ such that $\alpha_1=-a^2, \alpha_2=-b^2$, i.e., the equation of $E$ is
%$$y^2=(x+a^2)(x+b^2)x$$
%Replacing $E$ by $E(b)$ and putting $\lambda=a/b$,
The same arguments as in the proof of Theorem \ref{family8} allow us to assume
 that
$$E=\mathcal{E}_{1,\lambda}:y^2=(x+\lambda^2)(x+1)x.$$
Since $W_3$ is divisible by 4 in $\mathcal{E}_{1,\lambda}(K)$ and  all points of order dividing 4 lie in $\mathcal{E}_{1,\lambda}(K)$,
every point $R$ of  $\mathcal{E}_{1,\lambda}$ with $4R=W_3$ also lies in $\mathcal{E}_{1,\lambda}(K)$.
It follows from Proposition \ref{div4W3} that $K$ contains  $\mathbf{i}=\sqrt{-1}$ and there exist
$$c,d \in K\setminus \{ 0,1,-1, \pm 1\pm \sqrt{2},  \pm \sqrt{-1}\}$$
such that
$$\lambda= \left[\frac{c-\frac{1}{c}}{2}\right]^2, \ -\lambda= \left[\frac{d-\frac{1}{d}}{2}\right]^2.$$
This implies that
$$c-\frac{1}{c}=\pm \mathbf{i} \left (d-\frac{1}{d}\right).$$
Replacing if necessary $d$ by $-d$, we obtain the desired
$$c-\frac{1}{c}= \mathbf{i} \left (d-\frac{1}{d}\right).$$
\end{proof}

%%%%%%%%%%%%%%%%%%

\section{Points of order 3}
\label{l6}
The following assertion gives a simple description of points of order 3 on elliptic curves.

\begin{prop}
\label{order3}

 A point $P=(x_0,y_0)\in E(K)$ has order $3$ if and only if
 % all $x_0-\alpha_i$ are squares in $K$ and
 one can choose three square roots $r_i=\sqrt{x_0-\alpha_i}$ in such a way that
$$r_1 r_2 +r_2 r_3+r_3 r_1=0.$$
\end{prop}

\begin{proof}
Indeed, let  $P$ be a point of order 3.  Then $2(-P)=P$.
%In particular, if $P$ has order 3
Hence,  all $x_0-\alpha_i$ are squares in $K$.  By (\ref{x1}),
$$x(-P)=x_0+(r_1 r_2+r_2 r_3+r_3 r_1)$$
for a suitable choice of  $r_1,r_2,r_3$. Since
$x(-P)=x(P)=x_0$, we get $r_1 r_2 +r_2 r_3+r_3 r_1=0$.

Conversely, suppose that there exists a triple of square roots $r_i=\sqrt{x_0-\alpha_i}$ such that $r_1 r_2 +r_2 r_3+r_3 r_1=0$. Since $P\in E(K)$,
$$(r_1 r_2 r_3)^2=(x_0-\alpha_1)(x_0-\alpha_2)(x_0-\alpha_3)=y_0^2,$$
i.e., $r_1 r_2 r_3=\pm y_0$. Replacing $r_1,r_2,r_3$ by $-r_1,-r_2,-r_3$ if necessary, we may  assume that  $r_1 r_2 r_3=-y_0$. Then there exists a point $Q=(x(Q),y(Q))\in E(K)$ such that $2Q=P$, and
$x_1=x(Q), y_1=y(Q)$ are expressed in terms of $r_1,r_2,r_3$ as in (\ref{halfP}). Therefore
$$x(Q)=x_0+(r_1 r_2 +r_2 r_3+r_3 r_1)=x_0,$$
$$y(Q)=-y_0-(r_1+r_2+r_3)(r_1 r_2 +r_2 r_3+r_3 r_1)=-y_0,$$
i.e., $Q=-P, 2(-P)=P$, and so $P$ has order $3$.
\end{proof}

\begin{thm}
\label{fam3}
Let $a_1,a_2,a_3$ be   elements of $K$ such that
all $a_1^2, a_2^2,a_3^2$ are   distinct.
Let us consider the elliptic curve
$$E=E_{a_1,a_2,a_3}: y^2=(x+a_1^2)(x+a_2^2)(x+a_3^2)$$
over $K$. Let
$ P=(0,a_1 a_2 a_3) .$
Then $P$ enjoys the following properties.

\begin{itemize}
\item[(i)]
$P$ is divisible by $2$ in $E(K)$. More precisely, there are four points $Q\in E(K)$ with $2Q=P$,
namely,
\begin{equation*}\begin{aligned}
&(a_2 a_3-a_1 a_2-a_3 a_1, (a_1-a_2)(a_2+a_3)(a_3-a_1)), \\
&(a_3 a_1-a_1 a_2-a_2 a_3,  (a_1-a_2)(a_2-a_3)(a_3+a_1)),\\
&(a_1 a_2-a_2 a_3-a_3 a_1,(a_1+a_2)(a_2-a_3)(a_3-a_1), \\
 &(a_1 a_2+a_2 a_3+a_3 a_1, (a_1+a_2)(a_2+a_3)(a_3+a_1)).
\end{aligned}
\end{equation*}
\item[(ii)]
The following conditions are equivalent.
\begin{enumerate}
\item
$P$ has order $3$.
\item
None of $a_i$ vanishes, i.e.,
$\pm a_1, \pm a_2,\pm a_3$ are six distinct elements of $K$,
 and one of the following four equalities holds:
\begin{equation*}
\begin{aligned}
&a_2 a_3=a_1 a_2+a_3 a_1, \  a_3 a_1=a_1 a_2+a_2 a_3, \\
 &a_1 a_2=a_2 a_3+a_3 a_1, \ a_1 a_2+a_2 a_3+a_3 a_1=0.
\end{aligned}
\end{equation*}
\end{enumerate}
\item[(iii)]
Suppose that  equivalent conditions $(i)-(ii)$ hold. Then one of  four points $Q$ coincides with $-Q$ and has order 3, while the three other  points are of order $6$.
In addition, $E(K)$ contains a subgroup isomorphic to $\mathbb{Z}/6\mathbb{Z}\oplus \mathbb{Z}/2\mathbb{Z}$.
\end{itemize}
\end{thm}

\begin{rem}
Clearly, $E_{a_1,a_2,a_3}=E_{\pm a_1, \pm a_2, \pm a_3}$.
\end{rem}

\begin{proof}[Proof of Theorem \ref{fam3}]
 We have
$$\alpha_1=-a_1^2,\ \alpha_2=-a_2^2,\ \alpha_3=-a_3^2.$$
Let us try to divide $P$ by 2 in $E(K)$. We have
$$r_1=\pm a_1, \ r_2=\pm a_2,\ r_3=\pm a_3.$$
Since all $r_i$ lie in $K$, the point $P=(0,a_1 a_2 a_3)$ is divisible by 2 in $E(K)$.
Let $Q$ be a point on $E$ with $2Q=P$. By (\ref{x1}) and (\ref{chap}),
$$x(Q)=r_1 r_2+r_2 r_3+r_3 r_1, \ y(Q)=-(r_1+r_2)(r_2+r_3)(r_3+r_1)$$
with $r_1 r_2 r_3=-a_1 a_2 a_3$.  Plugging  in  $r_i= \pm a_i$
into the formulas for $x(Q)$ and $y(Q)$, we get explicit formulas for points $Q$ as in the statement of the theorem.
This proves (i).

Let us prove (ii).
Suppose that $P$ has order 3. Since $P$ is not of order $2$, we have
$0=x(P)\ne \alpha_i$ for all $i=1,2,3$.  Since
$$\{\alpha_1,\alpha_2,\alpha_3\}= \{-a_1^2,-a_2^2,-a_3^2\},$$
none of  $a_i$ vanishes. It follows from Proposition \ref{order3}
that one may choose the signs for $r_i$ in such a way that
$r_1 r_2+r_2 r_3+r_3 r_1=0$.
Plugging
%%%%%%%%
in
%%%%%%%%%%%%%%
$r_i=\pm a_i$  into this formula, we get four relations between $a_1,a_2,a_3$ as in (ii)(2).

Now suppose that one of the relations as in (ii)(2) holds. This means that one may choose the signs
of $r_i=\pm a_i$ in such a way that $r_1 r_2+r_2 r_3+r_3 r_1=0$. It follows from  Proposition \ref{order3}
that $P$ has order $3$. This proves (ii).

Let us prove (iii). Since $P$ has order $3$, $2(-P)=P$, i.e., $-P$ is one of the four $Q$'s.
Suppose that $Q$ is a point of $E$ with
$2 {Q}=P, \  {Q}\ne -P$.
Clearly, the order
%%%%%%
of
%%%%%%%%%%%
$ {Q}$ is either 3 or 6. Assume that  $ {Q}$ has order $3$. Then
$P=2 {Q}=- {Q}$
and therefore $ {Q}=-P$, which is not the case. Hence $ {Q}$ has order $6$.
Then $3 {Q}$ has order $2$, i.e., coincides with $W_i=(-a_i^2,0)$ for some $i \in \{1,2,3\}$.
Pick $j \in  \{1,2,3\}\setminus \{i\}$ and consider the  point $W_j=(-a_j^2,0)\ne W_i$.
Then the subgroup of $E(K)$ generated by  $ {Q}$ and $W_j$ is isomorphic to $\mathbb{Z}/6\mathbb{Z}\oplus \mathbb{Z}/2\mathbb{Z}$.
 This proves (iii).
\end{proof}

\begin{rem}
In Theorem \ref{fam3} we do {\sl not} assume that $\fchar(K)\ne {\it3}$!
\end{rem}

\begin{cor}
\label{fam3H}
Let $a_1,a_2,a_3$ be   elements of $K$ such that
$a_1^2, a_2^2,a_3^2$ are distinct.

Then the following conditions are equivalent.

\begin{itemize}
\item[(i)]
The point $P =(0,a_1 a_2 a_3)\in E_{a_1,a_2,a_3}(K)$ has order $3$.
\item[(ii)]
None of $a_i$ vanishes, and
one may choose signs for
$$a=\pm a_1,\ b=\pm a_2,\ c=\pm a_3$$
in such a way that
$c=ab/(a+b)$.
\end{itemize}

If these conditions hold, then
$$ E_{a_1,a_2,a_3}=E_{\lambda, b}: y^2=\left(x^2+(\lambda b)^2\right)\left(x+b^2\right)
\left(x+\left(\frac{\lambda}{\lambda+1} b\right)^2 \right),$$
where
$\lambda=a/b \in K\setminus \{ 0, \pm 1, -2,  -\frac{1}{2}\}.$
\end{cor}

\begin{proof}
Suppose  that condition   (ii) of the corollary holds, i.e.,  none of $a_i$ vanishes, and
one may choose signs for
$$a=\pm a_1,\ b=\pm a_2,\ c=\pm a_3$$
in such a way that
$c=ab/(a+b)$.  Then none of $a,b,c$ vanishes and
$ab=ac+bc$. By Theorem \ref{fam3}(ii),
$\mathcal{P}=(0,abc)$ is a point of order $3$ on the elliptic curve
$$E_{\lambda,b}= E_{a_1,a_2,a_3}.$$
Since $abc=\pm a_1 a_2 a_3$, either $\mathcal{P}=P$ or $\mathcal{P}=-P$.
In both cases $P$ has order $3$.

Notice that $\pm a_1, \pm a_2, \pm a_3$ are six distinct elements of $K$. This means
that $\pm a, \pm b, \pm c$ are also six distinct elements of $K$. If we put $\lambda=a/b$,
then
$$\pm \lambda b, \ \pm  b,\ \pm \frac{\lambda+1}{\lambda}b$$
are six  distinct elements of $K$.  This means (in light of the inequalities  $a\ne 0, b \ne 0$) that
$$\lambda \ne 0, \pm 1, -2, -\frac{1}{2}.$$
Suppose $P$ has order $3$.  By Theorem \ref{fam3}(ii), none of $a_i$ vanishes and one of the following four equalities holds:
\begin{equation*}\begin{aligned}&a_2 a_3=a_1 a_2+a_3 a_1, \  a_3 a_1=a_1 a_2+a_2 a_3,\\ & a_1 a_2=a_2 a_3+a_3 a_1, \ a_1 a_2+a_2 a_3+a_3 a_1=0.
\end{aligned}
\end{equation*}
Here are the corresponding  choices  of $a,b,c$ with $c=ab/(a+b)$:
\begin{equation*}\begin{aligned}
&a=a_1, b=-a_2, c= a_3; \ a=a_1, b=-a_2, c=a_3;\\
 &a=a_1, b=a_2, c=a_3; \  a=a_1, b=a_2, c=-a_3.
\end{aligned}
\end{equation*}
In order to finish the proof, we just need to notice that $a=\lambda b$ and
$$c=\frac{ab}{a+b}=\frac{\lambda b\cdot b}{\lambda b+b}=\frac{\lambda}{\lambda+1} b.$$
\end{proof}

\begin{thm}
\label{family3}
Let $E$ be an elliptic curve over $K$. Then $E(K)$ contains a subgroup isomorphic to  $\mathbb{Z}/6\mathbb{Z}\oplus \mathbb{Z}/2\mathbb{Z}$ if and only
if there exists $\lambda \in K \setminus \{ 0, \pm 1, -2,   -\frac{1}{2}\}$ such that $E$ is isomorphic to
$$\mathcal{E}_{3,\lambda}: y^2=\left(x^2+\lambda^2\right)(x+1)\left(x+\left(\frac{\lambda}{\lambda+1}\right)^2\right).$$
\end{thm}

\begin{proof}[Proof of Theorem  \ref{family3}]
Let $\lambda \in K \setminus \{ 0, \pm 1,  -2, -1/2\}$ and
 put $a_1=\lambda, a_2=1, a_3=\lambda/(\lambda+1)$. Then all $a_i$ do not vanish,
$a_1^2,a_2^2,a_3^2$ are three distinct elements of $K$,
$a_1 a_2=a_2 a_3+a_3 a_1$, and
$\mathcal{E}_{3,\lambda}=E_{a_1,a_2,a_3}$. It follows from Theorem \ref{fam3} that $\mathcal{E}_{3,\lambda}$
contains a subgroup isomorphic to $\mathbb{Z}/6\mathbb{Z}\oplus \mathbb{Z}/2\mathbb{Z}$.

Conversely, suppose that $E$ is an elliptic curve over $K$ such that $E(K)$ contains a subgroup isomorphic
to  $\mathbb{Z}/6\mathbb{Z}\oplus \mathbb{Z}/2\mathbb{Z}$. It follows that all three points of order $2$ lie in $E(K)$, and therefore
$E$ can be represented in the form \eqref{E2}.
It is also clear that $E(K)$ contains a point of order 3.
 Let us choose a point  $P=(x(P),y(P))\in E(K)$  of order 3.
We may assume that $x(P)=0$. We have $P=2(-P)$, and therefore $P$ is divisible by 2 in $E(K)$. By Theorem \ref{th0},
all $x(P)-\alpha_i=-\alpha_i$ are squares in $K$. This implies that there exist elements $a_1,a_2,a_3 \in K$ such that
$\alpha_i=-a_i^2$. Clearly, all three $a_1^2,a_2^2, a_3^2$ are distinct. Since $P$ lies on $E$,
$$y(P)^2=(x(P)+a_1^2)(x(P)+a_2^2)(x(P)+a_3^2)=a_1^2 a_2^2  a_3^2=(a_1 a_2 a_3)^2,$$
and therefore $y(P)= \pm a_1 a_2 a_3$. Replacing $P$ by $-P$ if necessary, we may   assume that $y(P)=a_1 a_2 a_3$,
i.e., $P=(0,a_1 a_2 a_3)$ is a $K$-point of order 3 on
$$E=E_{a_1,a_2,a_3}: y^2=(x+a_1)^2(x+a_2^2)(x+a_3)^2.$$
 It follows from Corollary \ref{fam3H} that there exist  {\sl  nonzero}
$b \in K$ and $\lambda \in K\setminus \{ 0, \pm 1,-2, -1/2\}$ such that
$$E=E_{a_1,a_2,a_3}=E_{\lambda,b}:y^2=\left(x+(\lambda b)^2\right)\left(x+b^2\right)\left(x+\left[\frac{\lambda}{\lambda+1} b\right]^2 \right).$$
But  $E_{\lambda,b}$ is isomorphic to
$$E_{\lambda,b}(b): {y^{\prime}}^2=(x^{\prime}+\lambda^2)(x^{\prime}+1)\left(x^{\prime}+\left[\frac{\lambda}{\lambda+1} \right]^2\right)$$
while the latter coincides with $\mathcal{E}_{3,\lambda}$.
%which is nothing else but $\mathcal{E}_{3,\lambda}$.
\end{proof}

%%%%%%%%%new stuff added on Jan. 22, 2017%%%%%%%%%%%%%%%%
\begin{rem}
There is a family of elliptic curves  over $\Q$ \cite[Table 3 on p.  217]{Kubert} (see also \cite[Appendix E]{Robledo}),
$$\mathfrak{E}_{3,t}:y^2+(1-a(t))xy -b(t)y=x^3-b(t)x^2,$$
with
$$a(t)=\frac{10-2t}{t^2-9}, \ b(t)=\frac{-2(t-1)^2(t-5)}{(t^2-9)^2}$$
(with $t\in \Q\setminus \{1, 5, \pm 3, 9\}$),
 whose group of rational points contains a subgroup isomorphic to $\mathbb{Z}/6\mathbb{Z}\oplus \mathbb{Z}/2\mathbb{Z}$.
 (The point $(0,0)$ of $\mathfrak{E}_{3,t}$  has order 6, ibid.)
 Let us assume that $t\ne \pm 3$ is an element  of an arbitrary field  $K$ (with $\fchar(K)\ne 2)$  and consider the
 cubic curve $\mathfrak{E}_{3,t}$ over $K$ defined by the same equation as above.

 In light of Theorem \ref{family3}, if $\mathfrak{E}_{3,t}$ is an elliptic curve over $K$,  then $\mathfrak{E}_{3,t}$ is isomorphic to $\mathcal{E}_{3,\lambda}$ for a certain $\lambda \in K$.
 Let us find the corresponding $\lambda$ (as a rational function of $t$).
  First, rewrite the equation for $\mathcal{E}_{3,\lambda}$ as
 $$\left(y+\frac{(1-a(t)x)-b(t)}{2}\right)^2=x^3-b(t)x^2+\left(\frac{(1-a(t))x-b(t)}{2}\right)^2.$$
 Second, multiplying the last equation by $(t^2-9)^6$ and introducing new variables
 $$y_1=(t^2-9)^3\cdot \left(y+\frac{(1-a(t))x-b(t)}{2}\right), \ x_1=(t^2-9)^2\cdot x,$$
 we obtain (with help of {\bf magma}) the equation for an isomorphic cubic curve
 $$\tilde{\mathfrak{E}}_{3,t}: y_1^2=(x_1-\alpha_1)(x_1-\alpha_2)(x_1-\alpha_3),$$
 where
 $$\alpha_1=-(2t^3-10t^2-18t+90)=-2(t-5)(t-3)(t+3),$$
 $$ \alpha_2=-(2t^3-10t^2+14t-6)=-2(t-3)(t-1)^2,$$
  $$ \alpha_3=-\left(\frac{1}{4}t^4-t^3-\frac{5}{2}t^2+7t-\frac{15}{4}\right)=-\frac{1}{4}(t-5)(t+3)(t-1)^2.$$
  We have
  $$\alpha_1-\alpha_2=-2^5(t-3), \ \alpha_2-\alpha_3=\frac{1}{4}\cdot  (t-1)^3(t-9), \ \alpha_3-\alpha_1=-\frac{1}{4}\cdot (t-5)^3(t+3).$$
  This implies that $\tilde{\mathfrak{E}}_{3,t}$ (and therefore $\mathfrak{E}_{3,t}$) is an elliptic curve over $K$ if and only if
  $$t \in K\setminus \{1,\pm 3, 5,9\}.$$
  Further we assume that this condition holds and therefore $\tilde{\mathfrak{E}}_{3,t}$ and  $\mathfrak{E}_{3,t}$ are elliptic curves over $K$.
  Clearly, all three points of order 2 on $\tilde{\mathfrak{E}}_{3,t}$ are defined over $K$ and the $K$-point
  $$Q=(x_1(Q), y_1(Q))=(0, -(t-5)(t-3)(t+3)(t-1)^2)$$ lies on $\tilde{\mathfrak{E}}_{3,t}$. We prove that $Q$ has order 6.
  %; it turns out that $Q$ has order $6$.
  % (Actually, it may be deduced from the description above
 % of the order of the point $(0,0)$ on $\mathfrak{E}_{3,t}$ but we give a direct proof.)
  Let us  consider the point
  $P=2Q\in E(K)$ with coordinates $x_1(P), y_1(P)\in K$. (Since $y_1(P)\ne 0$, $P\ne \infty$.) According to formulas of Section 1, there exists a unique
  triple ${r_1,r_2,r_3}$ of distinct  elements of $K$ such that
  $$(r_1+r_2)(r_2+r_3)(r_3+r_1)=-y_1(Q)=(t-5)(t-3)(t+3)(t-1)^2$$
  and for
  all $i=1,2,3$
  $$x_1(P)-\alpha_i=r_i^2,$$
  $$ \ 0 \ne -\alpha_i=x_1(Q)-\alpha_i=(r_i+r_j)(r_i+r_k),$$
  where $(i,j,k)$ is a permutation of $(1,2,3)$. This implies that
  $$r_1+r_2=\frac{(t-5)(t-3)(t+3)(t-1)^2}{-a_3}=\frac{(t-5)(t-3)(t+3)(t-1)^2}{\frac{1}{4}(t-5)(t+3)(t-1)^2}=4(t-3),$$
  $$r_2+r_3=\frac{(t-5)(t-3)(t+3)(t-1)^2}{-a_1}=\frac{(t-5)(t-3)(t+3)(t-1)^2}{2(t-5)(t-3)(t+3)}=\frac{1}{2}\cdot (t-1)^2,$$
  $$r_3+r_1=\frac{(t-5)(t-3)(t+3)(t-1)^2}{-a_2}=\frac{(t-5)(t-3)(t+3)(t-1)^2}{2(t-3)(t-1)^2}=\frac{1}{2}\cdot (t-5)(t+3).$$
  Hence
  $$r_1+r_2=4(t-3), \ r_2+r_3=\frac{(t-1)^2}{2}, \  r_3+r_1=\frac{(t+3)(t-5)}{2},$$
  and therefore
 $$r_1+r_2+r_3 =\frac{1}{2}\cdot\left((r_1+r_2)+(r_2+r_3)+(r_3+r_1)\right)=\frac{1}{2} \cdot (t^2+2t-19),$$
 which, in turn, implies that
 $$r_1=2t-10=2(t-5), \ r_2=2t-2=2(t-1), \ r_3=\frac{1}{2} \cdot (t-1)(t-5)=\frac{1}{8}r_1 r_2.$$
 One may easily check that
 $$c(t):=-2t^3+14t^2-22t+10=r_i^2+\alpha_i \ \text {for all}\ \ i=1,2,3.$$
 This implies that
 $$x_1(P)=c(t), \ c(t)-\alpha_i=r_i^2  \ \text {for all}\ \ i=1,2,3$$
 and $\tilde{\mathfrak{E}}_{3,t}$ is isomorphic to the elliptic curve
 $$E_{r_1,r_2,r_3}: y_1^2=(x_2+r_1^2)(x_2+r_2^2)(x_3+r_3^2)$$
 with $x_2=x_1-c(t)$. In addition,
 $$y_1(P)=-r_1 r_2 r_3=-2(t-1)^2(t-5).$$
 We have
 $$r_1 r_2=8r_3, \ r_2-r_1=8.$$
 This implies that $(r_2-r_1)r_3= r_1 r_2$,
which means that
 $$(-r_1) r_2+r_2 r_3+(-r_1) r_3=0.$$
 It follows from Proposition \ref{order3} that $P$ has order 3 in $\tilde{\mathfrak{E}}_{3,t}(K)$. (In particular, all $r_i \ne 0$.)
  Since $2Q=P$,
 the order of $Q$ in $\tilde{\mathfrak{E}}_{3,t}$ is 6.

 Notice that
 $$-r_3=\frac{(-r_1) r_2}{(-r_1)+r_2}$$
 and
 $$E_{r_1,r_2,r_3}=E_{-r_1,r_2,-r_3}.$$
 It follows from Corollary \ref{fam3H} and the end of the proof of Theorem  \ref{family3} that $E_{r_1,r_2,r_3}$ is isomorphic to $\mathcal{E}_{3,\lambda}$
 with
 $$\lambda=\frac{-r_1}{r_2}=\frac{-(2t-10)}{2t-2}=-\frac{t-5}{t-1}.$$
 This implies that $\mathfrak{E}_{3,t}$ is isomorphic to $\mathcal{E}_{3,\lambda}$ with $\lambda=-(t-5)/(t-1)$.
\end{rem}
%%%%%%%%%%%%%%%%%%%%%%%%%%%%%%%%%%%%%

\begin{cor}
Let $E$ be an elliptic curve over $\mathbb{F}_q$,  where $q=7,9,11,13$.
The group $E(\mathbb{F}_q)$ is isomorphic to $\mathbb{Z}/6\mathbb{Z}\oplus \mathbb{Z}/2\mathbb{Z}$ if and only if
$E$ is isomorphic to one of   elliptic curves
$\mathcal{E}_{3,\lambda}$.
%: y^2=\left(x^2+\lambda^2\right)(x+1)\left(x+\left[\frac{\lambda}{\lambda+1}\right]^2\right)$$
%with $\lambda \in \mathbb{F}_q \setminus \{ 0, \pm 1, -2,  -\frac{1}{2}\}$
\end{cor}

\begin{proof}
Suppose that $E(\mathbb{F}_q)$ is isomorphic to $\mathbb{Z}/6\mathbb{Z}\oplus \mathbb{Z}/2\mathbb{Z}$.
By Theorem \ref{family3}, $E$ is isomorphic to
one of  elliptic curves
$\mathcal{E}_{3,\lambda}$.
%: y^2=\left(x^2+\lambda^2\right)(x+1)\left(x+\left[\frac{\lambda}{\lambda+1}\right]^2\right)$$
%with $\lambda \in \mathbb{F}_q \setminus \{ 0, \pm 1, -2,  -\frac{1}{2}\}.$

Conversely, suppose that $E$ is isomorphic to one of these curves. We need to prove that
$E(\mathbb{F}_q)$ is isomorphic to $\mathbb{Z}/6\mathbb{Z}\oplus \mathbb{Z}/2\mathbb{Z}$.  By Theorem \ref{family3}, $E(\mathbb{F}_q)$ contains a subgroup isomorphic to $\mathbb{Z}/6\mathbb{Z}\oplus \mathbb{Z}/2\mathbb{Z}$; in particular, $12$ divides $|E(\mathbb{F}_q)|$. In order to finish the proof, it suffices to check that $|E(\mathbb{F}_q)|<24$, but this inequality follows from the Hasse bound \eqref{HasseB}
$$|E(\mathbb{F}_q)|\le q+2\sqrt{q}+1\le 13+2\sqrt{13}+1<22.$$
\end{proof}

%%%%%%%new stuff added on July 17, 2016%%%%
\begin{cor}
Let $E$ be an elliptic curve over $\mathbb{F}_{23}$.
The group $E(\mathbb{F}_{23})$ is isomorphic to $\mathbb{Z}/12\mathbb{Z}\oplus \mathbb{Z}/2\mathbb{Z}$ if and only if
$E$ is isomorphic to one of   elliptic curves
$\mathcal{E}_{3,\lambda}$.
%: y^2=\left(x^2+\lambda^2\right)(x+1)\left(x+\left[\frac{\lambda}{\lambda+1}\right]^2\right)$$
%with $\lambda \in \mathbb{F}_q \setminus \{ 0, \pm 1, -2,  -\frac{1}{2}\}$
\end{cor}

\begin{proof}
Suppose that $E(\mathbb{F}_{23})$ is isomorphic to $\mathbb{Z}/12\mathbb{Z}\oplus \mathbb{Z}/2\mathbb{Z}$. Then it contains a subgroup isomorphic to
 $\mathbb{Z}/6\mathbb{Z}\oplus \mathbb{Z}/2\mathbb{Z}$.
It follows from Theorem \ref{family3} that $E$ is isomorphic to
one of  elliptic curves
$\mathcal{E}_{3,\lambda}$.
%: y^2=\left(x^2+\lambda^2\right)(x+1)\left(x+\left[\frac{\lambda}{\lambda+1}\right]^2\right)$$
%with $\lambda \in \mathbb{F}_q \setminus \{ 0, \pm 1, -2,  -\frac{1}{2}\}.$

Conversely, suppose that $E$ is isomorphic to one of these curves. We need to prove that
$E(\mathbb{F}_{23})$ is isomorphic to $\mathbb{Z}/12\mathbb{Z}\oplus \mathbb{Z}/2\mathbb{Z}$.  By Theorem \ref{family3}, $E(\mathbb{F}_{23})$ contains a subgroup isomorphic to $\mathbb{Z}/6\mathbb{Z}\oplus \mathbb{Z}/2\mathbb{Z}$; in particular, $12$ divides $|E(\mathbb{F}_{23})|$. The Hasse bound \eqref{HasseB} tells us that
$$23+1-2\sqrt{23}\le |E(\mathbb{F}_{23})|\le 23+1+2\sqrt{23}$$
and therefore
$14<|E(\mathbb{F}_{23})|< 34$.
It follows that $|E(\mathbb{F}_{23})|=24$; in particular the 2-primary component  $E(\mathbb{F}_{23})(2)$ of  $E(\mathbb{F}_{23})$  has order 8.  On the other hand,  $E(\mathbb{F}_{23})(2)$ is isomorphic to a product of two cyclic groups, each of which has even order. This implies that  $E(\mathbb{F}_{23})(2)$ is isomorphic to $\mathbb{Z}/4\mathbb{Z}\oplus \mathbb{Z}/2\mathbb{Z}$. Taking into account that  $E(\mathbb{F}_{23})$  contains a point of order 3, we conclude that it contains a subgroup isomorphic to
$$(\mathbb{Z}/4\mathbb{Z}\oplus \mathbb{Z}/2\mathbb{Z})\oplus \mathbb{Z}/3\mathbb{Z}\cong
\mathbb{Z}/12\mathbb{Z}\oplus \mathbb{Z}/3\mathbb{Z}.$$
This subgroup has the same order 24 as the whole group  $E(\mathbb{F}_{23})$, which ends the proof.
\end{proof}

%%%%%%%%%%%%%new stuff added on January 15, 2017%%%%%%%%%%%%%%%%%%%%%
\begin{thm}
\label{Q6}
Let $K=\Q$  and $E$ an elliptic curve over $\Q$. Then the torsion subgroup $E(\Q)_t$ of $E(\Q)$ is isomorphic to
 $\mathbb{Z}/6\mathbb{Z}\oplus \mathbb{Z}/2\mathbb{Z}$
 if and only
if there exists $\lambda \in \Q \setminus \{ 0, \pm 1, -2,   -\frac{1}{2}\}$ such that $E$ is isomorphic to
$\mathcal{E}_{3,\lambda}$.
\end{thm}

\begin{proof}
By Theorem \ref{mazurQ} applied to $m=3$, if $E(\Q)$ contains a subgroup isomorphic
to  $\mathbb{Z}/6\mathbb{Z}\oplus \mathbb{Z}/2\mathbb{Z}$ then $E(\Q)_t$ is isomorphic
to  $\mathbb{Z}/6\mathbb{Z}\oplus \mathbb{Z}/2\mathbb{Z}$. Now the desired result follows from Theorem \ref{family3}.
\end{proof}

\section{Points of order $5$}
\label{l5}

The following assertion gives a  description of points of order $5$ on elliptic curves.
\begin{prop}\label{order5}
Let $P=(x_0,y_0)\in E(K)$.
The point  $P$ has order  $5$ if and only if, for any permutation $i,j,k$ of $1,2,3$,
one can choose square roots $r_i=\sqrt{x_0-\alpha_i}$ and  $r_i^{(1)}=\sqrt{(r_i+r_j)(r_i+r_k)}$ in such a way that
\begin{equation}
\label{ORDER5}
\begin{aligned}
(r_1 r_2+r_2 r_3+r_3 r_1)+(r_1^{(1)} r_2^{(1)}+r_2 ^{(1)}r_3^{(1)}+r_3^{(1)} r_1^{(1)})=0,\\
r_1 r_2+r_2 r_3+r_3 r_1\ne 0. \end{aligned}  \end{equation}

\end{prop}

\begin{rem}
Notice that if we drop the condition $r_1r_2r_3=-y_0$ in formulas \eqref{x1} and \eqref{chap}, then we get $8$ points $Q$ such that  $2Q=\pm P$. Similarly, if we drop the conditions
$r_1r_2r_3=-y_0$, $r_1^{(1)} r_2^{(1)} r_3^{(1)}=(r_1+r_2)(r_2+r_3)(r_3+r_1)$ in the formulas \eqref{1/4}, then we obtain all points $R$ for which $4R=\pm P$.
\end{rem}
\begin{proof}[Proof of Proposition \ref{order5}]
Suppose that $P$ has  order $5$. Then $-P$ is a $1/4$th of $P$. Therefore
there exist $r_i$ and $r_j^{(1)}$ such that
$$x(-P)=x(P)+(r_1 r_2+r_2 r_3+r_3 r_1)+(r_1^{(1)} r_2^{(1)}+r_2 ^{(1)}r_3^{(1)}+r_3^{(1)} r_1^{(1)}).$$
Since $x(P)=x(-P)$, we have
$$(r_1 r_2+r_2 r_3+r_3 r_1)+(r_1^{(1)} r_2^{(1)}+r_2 ^{(1)}r_3^{(1)}+r_3^{(1)} r_1^{(1)})=0.$$
On the other hand, if $r_1 r_2+r_2 r_3+r_3 r_1$, then the corresponding $Q$ (with $2Q=P$) satisfies
$$x(Q)=x(P)+(r_1 r_2+r_2 r_3+r_3 r_1)=x(P)$$
and therefore $Q=P$ or $-P$.  Since $2Q=P$, either $P=2P$ or $Q=-P=-2Q$ has order $5$. Clearly, $P\ne 2P$. If $Q=-2Q$ then $Q$ has order dividing $3$, which is not true, because its order is $5$.  The contradiction obtained proves that
$r_1 r_2+r_2 r_3+r_3 r_1 \ne 0$.

Conversely, suppose there exist square roots
$$r_i=\sqrt{x_0-\alpha_i} \ \text{ and } \ r_i^{(1)}=\sqrt{(r_i+r_j)(r_i+r_k)}$$ that satisfy (\ref{ORDER5}).
Replacing if necessary all $r_i$ by $-r_i$, we may and will assume that $r_1 r_2 r_3=-y(P)$. Let $Q=(x(Q),y(Q))$ be the corresponding half of $P$ with
$x(Q)=x(P)+(r_1 r_2+r_2 r_3+r_3 r_1)$. Since $r_1 r_2+r_2 r_3+r_3 r_1\ne 0$, we have $x(Q)\ne x(P)$; in particular,
$Q \ne -P$.
 Replacing if necessary all $r_i^{(1)}$ by $r_i^{(1)}$ , we may and will assume that
 $$r_1^{(1)} r_2^{(1)} r_3^{(1)}=(r_1+r_2)(r_2+r_3)(r_3+r_1)=-y(Q).$$
 Let $R=(x(R),y(R))$ be the corresponding half of $Q$. Then $4R=2(2R)=2Q=P$ and
 $$x(R)=x(P)+(r_1 r_2+r_2 r_3+r_3 r_1)+(r_1^{(1)} r_2^{(1)}+r_2 ^{(1)}r_3^{(1)}+r_3^{(1)} r_1^{(1)})=x(P).$$
 This means that either $R=P$ or $R=-P$.  If $R=P$, then $R=4R$ and $R$ has order $3$. This implies that both $Q=2R$ and $P=4R$ also have order $3$. It follows that  $P=-Q$, which is not the case.  Therefore
  $R=-P$. This means that $R=-4R$, i.e., $R$ has order $5$ and therefore $P=-R$ also has order $5$.
\end{proof}

In what follows we will use the following identities in the polynomial ring $\mathbb{Z}[t_1,t_2,t_3]$ that could be checked either directly or by using {\bf magma}.
\begin{equation}
\label{M0}
\begin{aligned}
(-t_1^2+t_2^2+t_3^2) (t_1^2-t_2^2+t_3^2)+ (t_1^2-t_2^2+t_3^2) ( t_1^2+t_2^2-t_3^2)\\+( t_1^2+t_2^2-t_3^2)(-t_1^2+t_2^2+t_3^2)=\\
-(t_1+t_2+t_3)(-t_1+t_2+t_3)(t_1-t_2+t_3)(t_1+t_2-t_3),
\end{aligned}
\end{equation}

\begin{equation}
\label{M1}
\begin{aligned}
(-t_1^2+t_2^2+t_3^2) (t_1^2-t_2^2+t_3^2)+ (t_1^2-t_2^2+t_3^2) ( t_1^2+t_2^2-t_3^2)\\+( t_1^2+t_2^2-t_3^2)(-t_1^2+t_2^2+t_3^2)+
4t_1^2t_2t_3 +4t_1t_2^2 t_3+ 4t_1t_2t_3^2\\=
t_1^4+t_2^4+t_3^4-2t_1^2t_2^2-2t_2^2t_3^2-2t_1^2t_3^2-4t_1^2t_2t_3-4t_1t_2^2t_3
-4t_1t_2t_3^2\\=
(t_1+t_2+t_3)\left(t_1^3 +t_2^3+t_3^3-t_1^2t_2-t_1t_2^2-t_2^2t_3-t_2t_3^2-t_1^2t_3- t_1t_3^2-2t_1t_2t_3\right).
\end{aligned}
\end{equation}

\begin{thm}
\label{Ea1a2a3}
 Let $a_1, a_2, a_3$  be elements of $K$ such that  $ \pm a_1, \pm a_2, \pm a_3 $
are six distinct elements of $K$ and none of three elements
$$\beta_1=-a_1^2+a_2^2+a_3^2,  \beta_2=a_1^2-a_2^2+a_3^2, \beta_3=a_1^2+a_2^2-a_3^2$$
vanishes. Then the following conditions hold.

\begin{itemize}
\item[(i)] None of $a_i$ vanishes and
$\beta_1^2, \beta_2^2, \beta_3^2$ are three distinct elements of $K$.

\item[(ii)]
Let us consider an elliptic curve
$$E_{5;a_1,a_2,a_3}: y^2=\left(x+\frac{\beta_1^2}{4}\right)\left(x+\frac{\beta_2^2}{4}\right)\left(x+\frac{\beta_3^2}{4}\right)$$
with  $P=(0, -\beta_1 \beta_2 \beta_3/8) \in E_{5;a_1,a_2,a_3}(K)$.

Then $P$ enjoys the following properties.
\begin{enumerate}
\item
 $P \in 2 E_{5;a_1,a_2,a_3}(K)$.
 \item
 Assume that
 \begin{equation}
 \label{Aord5}
 \begin{aligned}
 a_1^3 +a_2^3+a_3^3-a_1^2a_2-a_1a_2^2-a_2^2a_3-a_2a_3^2-a_1^2a_3- a_1a_3^2-2a_1a_2a_3=0,\\
 (a_1+a_2+a_3)(a_1-a_2-a_3)(a_1+a_2-a_3)(a_1-a_2+a_3)\neq 0.\end{aligned} \end{equation}\end{enumerate}
  Then $P$ has order $5$. In addition, $E_{5;a_1,a_2,a_3}(K)$ contains a subgroup isomorphic to $\mathbb{Z}/10\mathbb{Z}\oplus \mathbb{Z}/2\mathbb{Z}$.
    \end{itemize}
 \end{thm}

\begin{proof}
(i)  Since $a_i \ne - a_i$, none of $a_i$ vanishes.  Let $i,j \in\{1,2,3\}$ be two distinct indices and $k \in\{1,2,3\}$ be the third one. Then
$$\beta_i-\beta_j=a_j^2-a_i^2 \ne 0, \ \beta_i+\beta_j=2 a_k^2 \ne 0.$$
This implies that $\beta_i^2 \ne \beta_j^2$.

(ii)
Keeping our notation, we obtain that
$$r_1=\pm \frac{\beta_1}{2}=\pm\frac{-a_1^2+a_2^2+a_3^2}{2}, r_2=\pm \frac{\beta_2}{2}=\frac{a_1^2-a_2^2+a_3^2}{2},
r_3=\pm \frac{\beta_3}{2}=\pm \frac{a_1^2+a_2^2-a_3^2}{2},$$
$$r_i^{(1)}=\pm\sqrt{(r_i+r_j)(r_i+r_k)}$$
where $i,j,k$ is any permutation of $1,2,3$.
Thanks to Proposition \ref{order5}, it suffices to check that
  one may choose the  square roots
 $r_i$ and
$r_i^{(1)}$ in such a way that $r_1 r_2+r_2 r_3+r_3 r_1\ne 0$ and
\begin{equation}
\label{ast}
(r_1 r_2+r_2 r_3+r_3 r_1)+(r_1^{(1)} r_2^{(1)}+r_2 ^{(1)}r_3^{(1)}+r_3^{(1)} r_1^{(1)})=0.
\end{equation}
Let us put
$$r_i=\frac{\beta_i}{2}=\frac{-a_i^2+a_j^2+a_k^2}{2}.$$  We have
$$r_1+r_2=a_3^2,\quad r_1+r_3=a_2^2,\quad r_2+r_3=a_1^2. $$
It follows that
$$(r_1^{(1)})^2=a_2^2a_3^2,\quad (r_2^{(1)})^2=a_1^2a_3^2,\quad (r_3^{(1)})^2=a_1^2a_1^2.$$
Let us put
 $$r_1^{(1)}=a_2a_3, \quad r_2^{(1)}=a_1a_3,\quad r_3^{(1)}=a_1a_2.$$
Then the condition (\ref{ast}) may be rewritten as follows.
$$\begin{aligned}&(-a_1^2+a_2^2+a_3^2) (a_1^2-a_2^2+a_3^2)+ (a_1^2-a_2^2+a_3^2) ( a_1^2+a_2^2-a_3^2)\\+&( a_1^2+a_2^2-a_3^2)(-a_1^2+a_2^2+a_3^2)+
4a_1^2a_2a_3 +4a_1a_2^2 a_3+ 4a_1a_2a_3^2=0.
\end{aligned}$$
%%%%%%%%%%
\begin{comment}
 i.e.,
\begin{equation}
\label{doubleAST}
a_1^4+a_2^4+a_3^4-2a_1^2a_2^2-2a_2^2a_3^2-2a_1^2a_3^2-4a_1^2a_2a_3-4a_1a_2^2a_3
-4a_1a_2a_3^2=0.\end{equation}
The left hand side of (\ref{doubleAST}) splits into a product
$$(a_1+a_2+a_3)(a_1^3 +a_2^3+a_3^3-a_1^2a_2-a_1a_2^2-a_2^2a_3-a_2a_3^2-a_1^2a_3- a_1a_3^2-2a_1a_2a_3)$$
\end{comment}
%%%%%%%%%%%%%%%%%%
In light of (\ref{M1}),
 the condition (\ref{ast}) may be rewritten as
$$(a_1+a_2+a_3)(a_1^3 +a_2^3+a_3^3-a_1^2a_2-a_1a_2^2-a_2^2a_3-a_2a_3^2-a_1^2a_3- a_1a_3^2-2a_1a_2a_3)=0.$$
The latter equality follows readily from  the assumption (\ref{Aord5}) of Theorem.
By Proposition  \ref{order5}, it suffices to check that
$r_1r_2+r_2r_3+r_3r_1\ne 0$.
 In other words,  we need to prove that

\begin{equation}
\label{r1r2r3}
\begin{aligned}(-a_1^2+a_2^2+a_3^2)(a_1^2-a_2^2+a_3^2)+(a_1^2-a_2^2+a_3^2) ( a_1^2+a_2^2-a_3^2)\\+
 ( a_1^2+a_2^2-a_3^2)(-a_1^2+a_2^2+a_3^2)\ne 0.\end{aligned} \end{equation}
In light of (\ref{M0}), this inequality is equivalent to
%%%%%%%
\begin{comment}
 Opening brackets in (\ref{r1r2r3}) we get the equivalent inequality
 $$a_1^4+a_2^4+a_3^4-2a_1^2a_2^2-2a_2^2a_3^2-2a_1^2a_3^2\ne 0,$$
whose right hand side splits into a product and we get the equivalent inequality
\end{comment}
%%%%%%%%%%%%%%%%%%%%%
$$ (a_1+a_2+a_3)(a_1-a_2-a_3)(a_1+a_2-a_3)(a_1-a_2+a_3) \ne0.$$
But the latter inequality holds, by the assumption (\ref{Aord5}) of the theorem. Hence, $P$ has order $5$. Clearly, $P$ and all points of order $2$ generate a subgroup that is isomorphic to $\mathbb{Z}/10\mathbb{Z}\oplus \mathbb{Z}/2\mathbb{Z}$.
\end{proof}

\begin{thm}
\label{twist5}
Let $E$ be an elliptic curve over $K$. Then the following conditions are equivalent.
\begin{itemize}
\item[(i)]
$E(K)$ contains a subgroup isomorphic to $\mathbb{Z}/10\mathbb{Z}\oplus \mathbb{Z}/2\mathbb{Z}$.
\begin{comment}
There is an elliptic curve  $E^{\prime}$ over $K$ that is isomorphic to  $E$ either
over $K$ or over a quadratic extension of $K$ and
such that $E^{\prime}(K)$ contains a subgroup
isomorphic to $\mathbb{Z}/10\mathbb{Z}\oplus \mathbb{Z}/2\mathbb{Z}$.
\end{comment}
\item[(ii)]
There exists a triple $\{a_1,a_2,a_3\}\subset K$ that satisfies all the conditions of Theorem \ref{Ea1a2a3},
including (\ref{Aord5}), and
such that $E$ is isomorphic to  $E_{5;a_1,a_2,a_3}$.
%either over $K$ or over a quadratic extension of $K$.
\end{itemize}
\end{thm}

\begin{proof}
(i) follows from (ii), thanks to Theorem \ref{Ea1a2a3}.

Suppose (i) holds. In order to prove (ii) it suffices  to check that $E$  is isomorphic to  a
certain $E_{5;a_1,a_2,a_3}$   over $K$. We may assume that $E$ is defined by an equation of the form \eqref{E2}. Suppose that
$P=(0,y(P))\in E(K)$ has order $5$. Then
$P=4(-P)$ is divisible by $4$ in $E(K)$. This implies the existence of square roots
$r_i=\sqrt{-\alpha_i}\in K$ and  $r_i^{(1)}=\sqrt{(r_i+r_j)(r_i+r_k)}\in K$ in such a way that
$$
x(-P)=x(P)+(r_1 r_2+r_2 r_3+r_3 r_1)+(r_1^{(1)} r_2^{(1)}+r_2 ^{(1)}r_3^{(1)}+r_3^{(1)} r_1^{(1)}),$$
$$r_1^{(1)} r_2^{(1) }r_3^{(1)}=(r_1+r_2)(r_2+r_3)(r_3+r_1).$$
Since $x(-P)=x(P)=0$,
\begin{equation}
\label{R5}
(r_1 r_2+r_2 r_3+r_3 r_1)+(r_1^{(1)} r_2^{(1)}+r_2 ^{(1)}r_3^{(1)}+r_3^{(1)} r_1^{(1)})=0.
\end{equation}
Since the order of $P$ is {\sl not} $3$,
\begin{equation}
\label{NOT3}
r_1 r_2+r_2 r_3+r_3 r_1\ne 0.
\end{equation}
Recall that none of $r_i+r_j$ vanishes.  Let us choose square roots
$$b_1=\sqrt{r_2+r_3},  b_2=\sqrt{r_1+r_3}, b_3=\sqrt{r_1+r_2}$$
in such a way that
$r_1^{(1)} = b_2 b_3, r_2^{(1)} = b_3 b_1$.
Since
$$r_1^{(1)} r_2^{(1) }r_3^{(1)}=  b_1^2b_2^2b_3^2=(b_1b_2 b_3)^2,$$
we conclude that
$$r_3^{(1)} =\frac{r_1^{(1)} r_2^{(1) }r_3^{(1)}}{r_2^{(1)} r_3^{(1) }}=\frac{(b_1b_2 b_3)^2}{(b_2b_3)(b_3b_1)}=b_1b_2.$$ We obtain that
\begin{equation}
\label{b1b2b3}
r_1^{(1)} = b_2 b_3, r_2^{(1)} = b_3 b_1, r_3^{(1)}=b_1b_2.
\end{equation}
%(such a choice is possible, because the product of all $r_i^{(1)}$'s coincides with $(b_1b_2 b_3)^2$).
Unfortunately, $b_i$ do not have to lie in $K$. However,
 all the ratios $b_i/b_j$ lie in $K^{*}$.
 % and therefore $$b_i \in K(b_3)=K(\sqrt{r_1+r_2}).$$
We have
$$r_2+r_3=b_1^2, r_1+r_3=b_2^2, r_1+r_2=b_3^2$$
and therefore
\begin{equation}
\label{r5F}
\begin{aligned}
&r_1=\frac{-b_1^2+b_2^2+b_3^2}{2},  \ r_2=\frac{b_1^2-b_2^2+b_3^2}{2}, \ r_3=\frac{b_1^2+b_2^2-b_3^2}{2},\\
&\alpha_1=-r_1^2=\frac{(-b_1^2+b_2^2+b_3^2)^2}{4},  \ \alpha_2=-r_2^2=-\frac{(b_1^2-b_2^2+b_3^2)^2}{4},\\
&\alpha_3=-r_3^2=-\frac{(b_1^2+b_2^2-b_3^2)^2}{4}, \\
&P=(0, -(r_1+r_2)(r_2+r_3)(r_3+r_1))=(0, -b_1^2 b_2^2 b_3^2)\in E(K).
\end{aligned}
\end{equation}
Since none of $r_i$ vanishes, we get
$$-b_1^2+b_2^2+b_3^2\ne 0,  \ b_1^2-b_2^2+b_3^2\ne 0, \ b_1^2+b_2^2-b_3^2\ne 0.$$
Let us put
$$\gamma_1=-b_1^2+b_2^2+b_3^2,  \gamma_2=b_1^2-b_2^2+b_3^2, \gamma_3=b_1^2+b_2^2-b_3^2.$$
It follows from Theorem  \ref{Ea1a2a3}(i)  that all $\beta_i$ are {\sl distinct} nonzero elements of $K$.
The inequality (\ref{NOT3}) combined with first formula of (\ref{r5F}) gives us
\begin{equation*}
\begin{aligned}(-b_1^2+b_2^2+b_3^2)(b_1^2-b_2^2+b_3^2)+(b_1^2-b_2^2+b_3^2) ( b_1^2+b_2^2-b_3^2)\\+
 ( b_1^2+b_2^2-b_3^2)(-b_1^2+b_2^2+b_3^2)\ne 0,\end{aligned} \end{equation*}
 which is equivalent (thanks to (\ref{M0}))  to
 $$ (b_1+b_2+b_3)(b_1-b_2-b_3)(b_1+b_2-b_3)(b_1-b_2+b_3)\ne 0.$$
 In particular,
 $$b_1+b_2+b_3 \ne 0.$$
 The equality (\ref{R5}) gives us (thanks to (\ref{M1}))
 $$(b_1+b_2+b_3)(b_1^3 +b_2^3+b_3^3-b_1^2b_2-b_1b_2^2-a_2^2b_3-b_2b_3^2-b_1^2b_3- b_1b_3^2-2b_1b_2b_3)=0,$$
 i.e.,
 $$(b_1^3 +b_2^3+b_3^3-b_1^2b_2-b_1b_2^2-a_2^2b_3-b_2b_3^2-b_1^2b_3- b_1b_3^2-2b_1b_2b_3)=0.$$
 Let us put
 $$a_1=\frac{b_1}{b_3}, \ a_2=\frac{b_2}{b_3}, \ a_3=\frac{b_3}{b_3}=1.$$
 All $a_i$ lie in $K$.
 Clearly, the triple $\{a_1,a_2,a_3\}$ satisfies all the conditions of  Theorem \ref{Ea1a2a3}
including (\ref{Aord5}).
Let us put
\begin{equation*}
\begin{aligned}
\beta_1=-a_1^2+a_2^2+a_3^2=\frac{\gamma_1}{b_3^2}=\frac{\gamma_1}{r_1+r_2},\\
   \beta_2=a_1^2-a_2^2+a_3^2=\frac{\gamma_2}{b_3^2}=\frac{\gamma_2}{r_1+r_2},  \\
   \beta_3=a_1^2+a_2^2-a_3^2=\frac{\gamma_3}{b_3^2}=\frac{\gamma_3}{r_1+r_2}.
   \end{aligned}
   \end{equation*}

%All $\beta_i$ lie in $K$.

 The equation of $E$  is
$$y^2=\left(x+\frac{\gamma_1^2}{4}\right)\left(x+\frac{\gamma_2^2}{4}\right)\left(x+\frac{\gamma_3^2}{4}\right).$$
Then $E$ is isomorphic to

\begin{equation*}
\begin{aligned}
E(r_1+r_2): {y^{\prime}}^2=\left(x^{\prime}+\frac{\gamma_1^2}{4(r_1+r_2)^2}\right)
\left(x^{\prime}+\frac{\gamma_2^2}{4(r_1+r_2)^2}\right)
\left(x^{\prime}+\frac{\gamma_3^2}{4(r_1+r_2)^2}\right)=\\
\left(x^{\prime}+\frac{\beta_1^2}{4}\right)\left(x^{\prime}+\frac{\gamma_2^2}{4}\right)\left(x^{\prime}+\frac{\gamma_3^2}{4}\right).
\end{aligned}
\end{equation*}
Clearly, $E(r_1+r_2)$ coincides with $E_{5;a_1,a_2,a_3}$.
\begin{comment}
Dividing the equation by $b_3^{12}=(r_1+r_2)^6$, we obtain that
the map
$$x,y \mapsto x^{\prime}=\frac{x}{r_1+r_2}, \ y^{\prime}=\frac{y}{r_1+r_2}$$ establishes a $K$-isomorphism between $E^{\prime}$ and an elliptic curve
$$\frac{{y^{\prime}}^2}{r_1+r_2}=\left(x+\frac{\beta_1^2}{4}\right)\left(x+\frac{\beta_2^2}{4}\right)\left(x+\frac{\beta_3^2}{4}\right).$$
One has only to notice that the latter curve is isomorphic to $E_{5;a_1,a_2,a_3}$ over $K(b_3)=K(\sqrt{r_1+r_2})$.
\end{comment}
\end{proof}

\begin{rem}
\label{RLM5}
Let $E_{5;a_1,a_2,a_3}$ be as in Theorem \ref{Ea1a2a3}. Clearly,
 $E_{5;a_1,a_2,a_3}(a_3)= E_{5;a_1/a_3,a_2/a_3,1}$. Let us put
 $\lambda=a_1/a_3, \mu=a_2/a_3$.  Then
 \begin{equation}
 \label{ellLM}
 \begin{aligned}
  E_{5;a_1/a_3,a_2/a_3,1}= E_{5;\lambda,\mu,1}:\\
  y^2=\left[x+\left(\frac{-\lambda^2+\mu^2+1}{2}\right)^2\right]
  \left[x+\left(\frac{\lambda^2-\mu^2+1}{2}\right)^2\right]\left[x+\left(\frac{\lambda^2+\mu^2-1}{2}\right)^2\right].
  \end{aligned}
 \end{equation}
 The equation of (isomorphic) $E_{5;\lambda,\mu,1}\left(\frac{\lambda^2+\mu^2-1}{2}\right)$ is as follows.
 \begin{equation}
 E_{5;\lambda,\mu,1}\left(\frac{\lambda^2+\mu^2-1}{2}\right):
  y^2=\left[x+ \left(\frac{1-\lambda^2+\mu^2}{\lambda^2+\mu^2-1}\right)^2 \right]
\left[x+ \left(\frac{ \lambda^2-\mu^2+1}{\lambda^2+\mu^2-1}\right)^2 \right](x+1).
 \end{equation}

 The conditions on $a_1,a_2,a_3$ may be rewritten in terms of $\lambda,\mu$ as follows.
\begin{equation}
 \label{LMU}
 \begin{aligned}
 \lambda^3+\mu^3-\lambda^2\mu-\lambda\mu^2-\lambda^2-2\lambda\mu-\mu^2-\lambda-\mu+1=0,\\
\lambda\pm\mu\neq \pm1,\ \lambda\neq0, \ \mu\neq0,\ \lambda\neq\pm\mu,\\
\lambda^2+\mu^2\neq1,\ \lambda^2-\mu^2\neq\pm1.
\end{aligned}\end{equation}
The equality (\ref{LMU}) is equivalent to
\begin{equation}
 \label{LMU1}
(\lambda+\mu)(\lambda-\mu)^2-(\lambda+\mu)^2-(\lambda+\mu)+1=0.
\end{equation}
Multiplying (\ref{LMU1}) by (non-vanishing) $(\lambda+\mu)$, we get the equivalent equation
\begin{equation}
 \label{LMU2}
(\lambda^2-\mu^2)^2-(\lambda+\mu)^3-(\lambda+\mu)^2+(\lambda+\mu)=0.
\end{equation}
Let us make the change of variables
$$\xi=\lambda+\mu, \eta=\lambda^2-\mu^2.$$
Then (\ref{LMU2}) may be rewritten as
\begin{equation}
 \label{curve5}
\eta^2=\xi(\xi^2+\xi-1),
\end{equation}
which is an (affine model of an) elliptic curve if $\fchar(K)\ne 5$ and a singular rational plane cubic (Cartesian leaf) if
$\fchar(K)=5$.
Since
\begin{equation}
\label{lmsquare}
\lambda^2+\mu^2=\frac{(\lambda+\mu)^2+(\lambda-\mu)^2}{2}=\frac{\xi^2+\frac{\eta^2}{\xi^2}}{2}=
\frac{\xi^2+\frac{\xi^2+\xi-1}{\xi}}{2}=\frac{\xi^3+\xi^2+\xi-1}{2\xi},
\end{equation}
 the only restrictions on $(\xi,\eta)$ besides the equality  (\ref{curve5})
are the inequalities
$$\xi(\xi^2+\xi-1)\ne 0,\pm 1;  \ \xi^3+\xi^2+\xi-1 \ne 2\xi,  \ \pm 1 \ne \frac{\eta}{\xi}=\sqrt{\frac{\xi(\xi^2+\xi-1)}{\xi^2}},$$
i.e.
\begin{equation}
\label{zapret5}
\xi \ne 0,\pm 1, \frac{-1\pm \sqrt{5}}{2}.
\end{equation}
This means that
\begin{equation}
\label{frorbidKSI}
(\xi,\eta) \not\in \{(0,0), (\pm1, \pm1), (\frac{-1\pm \sqrt{5}}{2},0)\}.
\end{equation}
In light of (\ref{lmsquare}), the equation (\ref{ellLM}) may be rewritten with coefficients being rational functions in $\xi,\eta$ (rather than $(\lambda,\mu)$) as follows.
\begin{equation*}
%\begin{aligned}
%%%%%correction madw on Jan. 27, 2017%%%%%%%%%
\mathcal{E}_{5,\xi,\eta}:y^2=\left[x+ \left(\frac{2(1-\eta)}{\xi^3+\xi^2+\xi-3}\right)^2 \right]
\left[x+ \left(\frac{2( \eta+1)}{\xi^3+\xi^2+\xi-3}\right)^2 \right](x+1).
\end{equation*}
\end{rem}

\begin{comment}
\begin{rem}
\label{RS5}
Let us consider the family of elliptic curves
\begin{equation}
 \mathcal{E}_{5,\lambda,\mu}:
 y^2=\left[x+ \left(\frac{1-\lambda^2+\mu^2}{\lambda^2+\mu^2-1}\right)^2 \right]
\left[x+ \left(\frac{ \lambda^2-\mu^2+1}{\lambda^2+\mu^2-1}\right)^2 \right](x+1).
\end{equation}
where $\lambda, \mu\in K$ satisfy
\begin{equation}
\label{restriction5}
\begin{aligned}
 (\lambda+\mu)(\lambda-\mu)^2-(\lambda+\mu)^2-(\lambda+\mu)+1=0,\\
\lambda,\mu\neq0,\pm1, \lambda\neq \pm\mu.
\end{aligned}\end{equation}
Then
$$\mathcal{E}_{5,\lambda,\mu}=E_{5;\lambda,\mu,1}\left(\frac{\lambda^2+\mu^2-1}{2}\right).$$
 \end{rem}
 \end{comment}
\begin{thm}
\label{family5}
Let $E$   be an elliptic curve over $K$. Then the following conditions are equivalent.
\begin{itemize}
\item[(i)]
%There is an elliptic curve  $E^{\prime}$ over $K$ that is isomorphic to  $E$ either
%over $K$ or over a quadratic extension of $K$ and
%such that
$E(K)$ contains a subgroup
isomorphic to $\mathbb{Z}/10\mathbb{Z}\oplus \mathbb{Z}/2\mathbb{Z}$.
\item[(ii)]
There exist $(\xi, \eta)\in K^2$ that satisfy the equation (\ref{curve5}) and inequalities (\ref{frorbidKSI}) and such that
$E$ is isomorphic to   $\mathcal{E}_{5,\xi,\eta}$.
% either over $K$ or over a quadratic extension of $K$.
\end{itemize}
\end{thm}

\begin{proof}
The result follows from Theorem \ref{twist5} combined with Remark \ref{RLM5}.
% and \ref{RS5}.
\end{proof}

\begin{rem}
In Theorem \ref{family5} we do {\sl not} assume that $\fchar(K)\ne{\it 5}!$
\end{rem}

\begin{cor}
Let $E$ be an elliptic curve over $\mathbb{F}_q$ with $q =13,17,19,23,25,27 $.  Then $E(\mathbb{F}_q)$ is isomorphic to $\mathbb{Z}/10\mathbb{Z}\oplus\mathbb{Z}/2\mathbb{Z}$ if and only if $E$ is isomorphic to one of $\mathcal{E}_{5,\xi,\eta}$.
\end{cor}

\begin{proof}
Suppose that $E(\mathbb{F}_q)$ is isomorphic to $\mathbb{Z}/10\mathbb{Z}\oplus \mathbb{Z}/2\mathbb{Z}$.
By Theorem \ref{family5}, $E$ is isomorphic to
one of  elliptic curves
$\mathcal{E}_{5,\xi,\eta}$.

Conversely, suppose that $E$ is isomorphic to one of these curves. We need to prove that
$E(\mathbb{F}_q)$ is isomorphic to $\mathbb{Z}/10\mathbb{Z}\oplus \mathbb{Z}/2\mathbb{Z}$.  By Theorem \ref{family5}, $E(\mathbb{F}_q)$ contains a subgroup isomorphic to $\mathbb{Z}/10\mathbb{Z}\oplus \mathbb{Z}/2\mathbb{Z}$; in particular, $20$ divides $|E(\mathbb{F}_q)|$. In order to finish the proof, it suffices to check that $|E(\mathbb{F}_q)|<40$, but this inequality follows from the Hasse bound \eqref{HasseB}
$$|E(\mathbb{F}_q)|\le q+2\sqrt{q}+1\le 27+2\sqrt{27}+1<40.$$
\end{proof}

\begin{cor}
 Let $E$ be an elliptic curve over $\mathbb{F}_q$ with $q=31,37,41,43$. Then $E(\mathbb{F}_q)$ is isomorphic to $\mathbb{Z}/20\mathbb{Z}\oplus\mathbb{Z}/2\mathbb{Z}$ if and only if $E$ is isomorphic to one of $\mathcal{E}_{5,\xi,\eta}$.
\end{cor}

\begin{proof}
Suppose that $E(\mathbb{F}_q)$ is isomorphic to
$\mathbb{Z}/20\mathbb{Z}\oplus \mathbb{Z}/2\mathbb{Z};$
 the latter contains a subgroup isomorphic to  $\mathbb{Z}/10\mathbb{Z}\oplus \mathbb{Z}/2\mathbb{Z}$.
By Theorem \ref{family5}, $E$ is isomorphic to
one of  elliptic curves
$\mathcal{E}_{5,\xi,\eta}$.

Conversely, suppose that $E$ is isomorphic to one of these curves. We need to prove that
$E(\mathbb{F}_q)$ is isomorphic to $\mathbb{Z}/20\mathbb{Z}\oplus \mathbb{Z}/2\mathbb{Z}$.  By Theorem \ref{family5}, $E(\mathbb{F}_q)$ contains a subgroup isomorphic to $\mathbb{Z}/10\mathbb{Z}\oplus \mathbb{Z}/2\mathbb{Z}$; in particular, $20$ divides $|E(\mathbb{F}_q)|$. It follows from the Hasse bound \eqref{HasseB} that
$$20<31-2\sqrt{31}+1 \le |E(\mathbb{F}_q)|\le 43+2\sqrt{43}+1<60.$$
This implies that $|E(\mathbb{F}_q)|=40$, and therefore $E(\mathbb{F}_q)$ is isomorphic to a direct sum of $\mathbb{Z}/5\mathbb{Z}$ and the order 8 abelian group $E(\mathbb{F}_q)(2)$; in addition, the latter group is isomorphic to a direct sum of two cyclic groups of even order (because it contains a subgroup isomorphic to $\mathbb{Z}/2\mathbb{Z}\oplus \mathbb{Z}/2\mathbb{Z}$). This implies that  $E(\mathbb{F}_q)(2)$ is isomorphic to $\mathbb{Z}/4\mathbb{Z}\oplus \mathbb{Z}/2\mathbb{Z}$. It follows that  $E(\mathbb{F}_q)$ is isomorphic to a direct sum
$$\mathbb{Z}/5\mathbb{Z}\oplus \mathbb{Z}/4\mathbb{Z}\oplus \mathbb{Z}/2\mathbb{Z}\cong
\mathbb{Z}/20\mathbb{Z}\oplus\mathbb{Z}/2\mathbb{Z}.$$
\end{proof}

\begin{cor}
 Let $E$ be an elliptic curve over $\mathbb{F}_q$ with $q =59 $ or $61$. Then $E(\mathbb{F}_q)$ is isomorphic to $\mathbb{Z}/30\mathbb{Z}\oplus\mathbb{Z}/2\mathbb{Z}$ if and only if $E$ is isomorphic to one of $\mathcal{E}_{5,\xi,\eta}$.
\end{cor}

\begin{proof}
Suppose that $E(\mathbb{F}_q)$ is isomorphic to
$\mathbb{Z}/30\mathbb{Z}\oplus \mathbb{Z}/2\mathbb{Z}$; the latter contains a subgroup isomorphic to  $\mathbb{Z}/10\mathbb{Z}\oplus \mathbb{Z}/2\mathbb{Z}$.
By Theorem \ref{family5}, $E$ is isomorphic to
one of  elliptic curves
$\mathcal{E}_{5,\xi,\eta}$.

Conversely, suppose that $E$ is isomorphic to one of these curves. We need to prove that
$E(\mathbb{F}_q)$ is isomorphic to $\mathbb{Z}/30\mathbb{Z}\oplus \mathbb{Z}/2\mathbb{Z}$.  By Theorem \ref{family5}, $E(\mathbb{F}_q)$ contains a subgroup isomorphic to $\mathbb{Z}/10\mathbb{Z}\oplus \mathbb{Z}/2\mathbb{Z}$; in particular, $20$ divides $|E(\mathbb{F}_q)|$. It follows from the Hasse bound \eqref{HasseB} that
$$40<59-2\sqrt{59}+1\le |E(\mathbb{F}_q)|< 61+2\sqrt{61}+1<80.$$
This implies that $|E(\mathbb{F}_q)|=60$; in particular, $E(\mathbb{F}_q)$ contains a subgroup isomorphic to $\mathbb{Z}/3\mathbb{Z}$. This implies that $E(\mathbb{F}_q)$ contains a subgroup isomorphic to
$$(\mathbb{Z}/10\mathbb{Z}\oplus \mathbb{Z}/2\mathbb{Z})\oplus \mathbb{Z}/3\mathbb{Z}\cong \mathbb{Z}/30\mathbb{Z}\oplus \mathbb{Z}/2\mathbb{Z};$$
the order of this subgroup is 60, i.e.,  it coincides with the order of the whole group $E(\mathbb{F}_q)$.
\end{proof}

%%%%%%%%%%%%%%%%%%%%%%%%%%new stuff added Jan. 17, 2017%%%%%%%%
\begin{thm}
\label{familyQu5}
Let $K$ be a quadratic field and $E$   be an elliptic curve over $K$. Then the following conditions are equivalent.
\begin{itemize}
\item[(i)]
The torsion subgroup $E(K)_t$ of
$E(K)$ is
isomorphic to $\mathbb{Z}/10\mathbb{Z}\oplus \mathbb{Z}/2\mathbb{Z}$.
\item[(ii)]
There exist $(\xi, \eta)\in K^2$ that satisfy the equation (\ref{curve5}) and inequalities (\ref{frorbidKSI}) and such that
$E$ is isomorphic to   $\mathcal{E}_{5,\xi,\eta}$.
% either over $K$ or over a quadratic extension of $K$.
\end{itemize}
\end{thm}

\begin{proof}
By Theorem \ref{Kquad},
if $E(K)$ contains a subgroup  isomorphic to  $\mathbb{Z}/10\mathbb{Z}\oplus \mathbb{Z}/2\mathbb{Z}$ then
$E(K)_t$  is isomorphic to  $\mathbb{Z}/10\mathbb{Z}\oplus \mathbb{Z}/2\mathbb{Z}$. Now the desired result follows from Theorem \ref{family5}.
\end{proof}


\begin{thebibliography}{99}

\bibitem{Bombieri} E. Bombieri, W. Gubler,
Heights in Diophantine Geometry.
New Mathematical Monographs, {\bf 4}, Cambridge University Press, Cambridge, 2006.

\bibitem{Buhler} J.P. Buhler, {\sl Elliptic curves, modular forms and applications}, pp. 5--81.  In: Arithmetic Algebraic Geometry (B. Conrad, K. Rubin, eds.) IAS/Park City Mathematics Series {\bf 9},  American Mathematical Society, Providence, RI, 2001.

\bibitem{Cassels} J.W.C. Cassels, {\sl Diophantine equations with special reference to elliptic curves}. J. London Math. Soc. {\bf 41} (1966), 193--291.

\bibitem{Husemoller} D. H\"{u}semoller, Elliptic Curves. Second edition. With appendices by O. Forster, R. Lawrence and S. Theisen.
Graduate Texts in Mathematics {\bf 111} Springer-Verlag, New York, 2004.

\bibitem{Kam} S. Kamienny, {\sl Torsion points on elliptic curves and q-coefficients of modular forms}.  Invent. Math. {\bf 109} (1992), 221--229.

\bibitem{KamN} S. Kamienny, F. Najman,
{\sl Torsion groups of elliptic curves over quadratic fields}.
Acta Arith. {\bf 152} (2012), no. 3, 291--305.

\bibitem{Ken} M. A. Kenku, F. Momose, {\sl Torsion points on elliptic curves defined over quadratic fields}. Nagoya Math. J. {\bf 109} (1988), 125--149.


 \bibitem{Knapp} A. Knapp, Elliptic Curves. Mathematical Notes {\bf 40},  Princeton University Press, Princeton, NJ, 1992.

\bibitem{Kubert} D.S. Kubert, {\sl Universal bounds on the torsion of elliptic curves}. Proc. London Math. Soc. (3) {\bf 33} (1976), 193--237.


 \bibitem{Lang} S.  Lang, Elliptic Curves: Diophantine Analysis.
 Grundlehren der Mathematischen Wissenschaften {\bf 231}, Springer-Verlag, Berlin-New York, 1978.





\bibitem{Robledo} A. Lozano-Robledo, Elliptic Curves, Modular Forms and their L-functions. Student Mathematical Library {\bf 38},
American Mathematical Society, Providence, RI, 2011.

% \bibitem{Manin} Yu.I. Manin, {\sl Cyclotomic fields and modular curves}. Uspekhi Mat. Nauk {\bf 26:6} (1971), 7--71;
% Russian Mathematical Surveys {\bf  26:6} (1971), 7--78.


\bibitem{Mazur} B. Mazur, {\sl Modular curves and the Eisenstein ideal}.
Inst. Hautes \'Etudes Sci. Publ. Math. No. 47 (1977), 33--186.

 \bibitem{Schaefer} E. Schaefer, $2$-descent on the Jacobians of hyperelliptic curves.
J. Number Theory {\bf 51} (1995), no. 2, 219--232.

\bibitem{SS} N. Schappacher, R.  Schoof, {\sl Beppo Levi and the arithmetic of elliptic curves}.
Mathematical Intelligencer {\bf 18} (1996), 57--69.

\bibitem{Silver} A. Silverberg, {\sl Explicit Families of Elliptic Curves with Prescribed Mod $N$ representations}, pp. 447--461.
 In: Modular Forms and Fermat's Last Theorem (G. Cornell, J.H. Silverman, G. Stevens, eds.) Springer-Verlag New York Inc., 1997.

 \bibitem{Silver2} A. Silverberg, {\sl Open questions in Arithmetic Algebraic Geometry}, pp. 83--142. In: Arithmetic Algebraic Geometry (B. Conrad, K. Rubin, eds),
 IAS/Park City Mathematics Series {\bf 9}, American Mathematical Society, Providence, RI, 2001.

\bibitem{Silverman} J.S. Silverman, Arithmetic of Elliptic Curves. Second edition.
Graduate Texts in Mathematics
{\bf 106}, Springer, Dordrecht Heidelberg London New York,  2009.

\bibitem{SilvermanTate} J.S. Silverman, J.T. Tate, Rational Points on Elliptic Curves. Second Edition. Springer Cham Heidelberg New York Dordrecht London, 2015.

\bibitem{Shioda}T. Shioda, {\sl On rational points of the generic elliptic curve with level $N$ structure
over the field of modular functions of level $N$}.  J. Math. Soc. Japan {\bf 25} (1973),
144-157.


\bibitem{Tate} J. Tate, {\sl Algebraic Formulas in Arbitrary Characteristic}.  Appendix 1 to:
S. Lang, Elliptic Functions, Second Edition. Springer-Verlag New York Inc., 1987.

\bibitem{Wash} L.C.  Washington, Elliptic Curves: Number Theory and Cryptography. Second edition.  Chapman \& Hall/CRC Press, Boca Raton London New York, 2008.


\bibitem{Yelton} J. Yelton, Dyadic torsion of elliptic curves.
European J. Math. {\bf 1} (2015), no. 4, 704--716.


\end{thebibliography}
\end{document}